\documentclass[a4paper,reqno]{amsart}

\usepackage[maxbibnames=9]{biblatex}

\usepackage{color}
\usepackage{graphicx}
\usepackage[normalem]{ulem}
\usepackage{morefloats}
\usepackage{hyperref}
\numberwithin{equation}{section}
\hypersetup{
    bookmarks=true,         % show bookmarks bar?
    unicode=false,          % non-Latin characters in Acrobat’s bookmarks
    pdftoolbar=true,        % show Acrobat’s toolbar?
    pdfmenubar=true,        % show Acrobat’s menu?
    pdffitwindow=false,     % window fit to page when opened
    pdfstartview={FitH},    % fits the width of the page to the window
    pdftitle={My title},    % title
    pdfauthor={Author},     % author
    pdfsubject={Subject},   % subject of the document
    pdfcreator={Creator},   % creator of the document
    pdfproducer={Producer}, % producer of the document
    pdfkeywords={keyword1, key2, key3}, % list of keywords
    pdfnewwindow=true,      % links in new PDF window
    colorlinks=false,       % false: boxed links; true: colored links
    linkcolor=green,          % color of internal links (change box color with linkbordercolor)
    citecolor=green,        % color of links to bibliography
    filecolor=green,      % color of file links
    urlcolor=green,           % color of external links
    urlbordercolor={1 1 1}  % color of border around links
}

\usepackage{amsmath, amsfonts, amsthm, amssymb}
\usepackage{fancyhdr}

\usepackage{biblatex}

\usepackage{mathrsfs}
\usepackage{blkarray}
\usepackage{amsthm}
\usepackage{mathtools}
\usepackage[title]{appendix}

\newtheorem{theorem}{Theorem}[section]
\newtheorem*{theorem*}{Theorem}
\newtheorem{lemma}[theorem]{Lemma}

\newcommand{\DD}{\mathbb{D}}

\newtheorem{proposition}[theorem]{Proposition}

\newtheorem{corollary}[theorem]{Corollary}
\theoremstyle{definition}
\newtheorem{definition}[theorem]{Definition}

\newtheorem{remark}[theorem]{Remark}

\theoremstyle{definition}
\newtheorem{exmp}[theorem]{Example} 
\theoremstyle{definition}

\begin{document}

\title[\small DENJOY-WOLFF POINTS]{\large DENJOY-WOLFF POINTS ON THE bidisK VIA MODELS}

\author[\small M. T.  JURY]{MICHAEL T. JURY}
\address{DEPARTMENT OF MATHEMATICS, UNIVERSITY OF FLORIDA, GAINESVILLE, FL}
\email{mjury@ufl.edu} 
\author[\small G. TSIKALAS]{GEORGIOS TSIKALAS}
\address{DEPARTMENT OF MATHEMATICS AND STATISTICS, WASHINGTON UNIVERSITY IN ST. LOUIS, ST. LOUIS, MO}
\email{gtsikalas@wustl.edu} 
\thanks{Jury partially supported by National Science Foundation Grant DMS 2154494. Tsikalas partially supported by National Science Foundation Grant DMS 2054199 and by Onassis Foundation - Scholarship ID: F ZR 061-1/2022-2023.  }

\small
\begin{abstract}
    \small
Let $F=(\phi, \psi):\DD^2\to\DD^2$ denote a holomorphic self-map of the bidisk without interior fixed points. It is well-known that, unlike the case with self-maps of the disk, the sequence of iterates $$\{F^n:=F\circ F\circ \cdots \circ F\}$$ needn't converge. The cluster set of $\{F^n\}$ was described in a classical 1954 paper of Herv\'{e}.  Motivated by Herv\'{e}'s work and the Hilbert space perspective of Agler, McCarthy and Young on boundary regularity, we propose a new approach to boundary points of Denjoy-Wolff type for the coordinate maps $\phi, \psi.$ We establish several equivalent descriptions of our Denjoy-Wolff points, some of which only involve checking specific directional derivatives and are particularly convenient for applications. Using these tools, we are able to refine Herv\'{e}'s theorem and show that, under the extra assumption of $\phi$ and $\psi$ possessing Denjoy-Wolff points with certain regularity properties, one can draw much stronger conclusions regarding the behavior of $\{F^n\}.$

\end{abstract}
\maketitle

\small
 \section{INTRODUCTION} \label{intro}
 \large 

Let $\DD$ denote the open unit disk. Given a holomorphic map $f:\DD\to\DD$ without fixed points, a theorem of Wolff \cite{Wolff3} states that there exists a boundary point $\tau\in\partial\DD$ such that every closed disk internally tangent to $\DD$ at $\tau$ (in other words, every horocycle containing $\tau$) is invariant under $f$. From this, one can deduce the classical Denjoy-Wolff Theorem \cite{Denjoy}, \cite{Wolff1}, \cite{Wolff2}: the sequence of iterates 
$$f^n:=\underbrace{f\circ f\circ \cdots \circ f}_\textrm{$n$ times}$$
converges to $\tau$ uniformly on compact subset of $\mathbb{D}.$ In this setting, the (unique) point $\tau$ is
termed the \textit{Denjoy-Wolff point} of $f$. See \cite{Burckel} for a nice exposition of the details and many historical remarks. \par 
A lot of work has been devoted to obtaining higher-dimensional generalizations of the Denjoy-Wolff Theorem. The first such result is due to Herv\'{e} \cite{Herveball}, who proved an exact analogue of the Denjoy-Wolff Theorem for fixed-point-free self-maps of the unit ball $\mathbb{B}_{n}\subset \mathbb{C}^{n}$ (see also \cite{Maccluer}). Later, Abate \cite{Abatestronglyconvex} (see also the excellent survey \cite{Abatebook}) achieved a generalization of this result to all smoothly bounded strongly convex domains in $\mathbb{C}^n,$ paving the way for further extensions to smoothly bounded pseudoconvex domains of both finite and infinite type (see \cite{TranVu} and the references therein). More recently, Budzy\'{n}ska \cite{BudCn} (see also \cite{Budcondensing} and \cite{BudDWtype}) showed that the smoothness assumption can be dropped if one restricts to strictly convex domains. \par 
Unfortunately, the situation becomes considerably more complicated in general bounded domains. The proofs of the above results utilize certain $f$-invariant domains (usually termed \textit{horospheres}, as they generalize Wolff's horocycles) which may have too large intersections with the boundary of the domain in the general case, making it difficult to control the behavior of the iterates. Indeed, even though several different types of horospheres have been considered in the literature with varying degrees of generality (see e.g. \cite{Abatebook}, \cite{AbateRaissynonsmoothconvex},  \cite{BudCn}, \cite{ChuRigbyhoroballs}, \cite{FrosiniBusemann}, \cite{Mellonfiniterank}, where the focus is either on bounded convex or bounded symmetric domains), boundary smoothness or extra convexity assumptions (or a mixture of both) are generally required to control the size of the intersection with the boundary. This is true even in very simple finite-dimensional domains, such the unit polydisk $\mathbb{D}^n,$ where the presence (for $n\ge 2$) of large ``flat" boundary components prevents the iterates from converging. In such a case, one seeks to understand the cluster points of $\{f^n\}.$ Although holomorphic dynamics on $\mathbb{D}^n$ (for general $n$) have been studied by a number of authors (see e.g. \cite{AbateJuliaWolff},\cite{AbateRaissynonsmoothconvex},  \cite{Valironabelpolydisc}, \cite{ChuRigbyHilbertballs}, \cite{FrosiniBusemann}, \cite{MellonDWboundesymmetric}), progress on iteration-theoretic questions remains limited. \par Somewhat stronger conclusions can be drawn if one restricts their attention to the bidisk. Let $F=(\phi, \psi): \DD^2\to\DD^2$ be holomorphic and without fixed points. The best known general results regarding the behavior of the iterates $\{F^n\}$ in this setting can be found in the classical paper \cite{HERVE} of Herv\'{e} (see also \cite{FrosiniBusemann}, \cite{FrosiniDynamics}, \cite{Nowell}, \cite{SolaTullydynamics} for more recent work concerning the bidisk). Herv\'{e} observed that all holomorphic maps $\phi: \DD^2\to\DD$ (that are not coordinate projections) can be classified into two separate categories (see Definition \ref{Type I, Type II}) based on the location of the Denjoy-Wolff points of the slice functions $\phi_{\mu}:\DD\to\DD$, where $\phi_{\mu}(\lambda)=\phi(\lambda, \mu)$ for all $\lambda, \mu\in\DD.$ He then gave a description of the cluster points of $\{F^n\}$ by considering three distinct cases (see Theorem \ref{HERVE's THEOREM}), depending on the categories that the coordinate functions $\phi$ and $\psi$ belong to. \cite{HERVE} also contains numerous examples demonstrating that, from a certain viewpoint, these results are optimal.
\par 
In the present work, motivated by the model-theoretic techniques of \cite{AMYcaratheodory} and \cite{Sarason}, we propose new definitions for Denjoy-Wolff-type points of holomorphic functions $\phi:\DD^2\to\DD$ (see Definition \ref{TypeI+IIdef}). These will be boundary points where $\phi$ satisfies a mild regularity condition (termed B-points following \cite{AMYcaratheodory}, see Section \ref{prelims} for definitions) and appropriate contractivity assumptions stated in terms of the model function. We prove several equivalent characterizations of our Denjoy-Wolff points, some of which are particularly easy to verify in practice and involve certain directional derivatives of $\phi$ at the points in question (see Theorems \ref{TypeIcharabridged}, \ref{TypeIIcharabridged} and \ref{ultimateDWchar}). This constitutes a departure from the usual criteria for Denjoy-Wolff points used in the setting of $\DD^2$, which depend on the existence of invariant horospheres. With these tools in our disposal, we are able to refine Herv\'{e}'s theorem. Among several results, we show that
if the coordinate functions $\phi$ and $\psi$ of $F$ possess certain Denjoy-Wolff points that are B- but not C-points (i.e. the functions do not have angular gradients there), then one gains much tighter control over the behavior of the iterates $\{F^n\}$(see Theorems \ref{TypeITypeIIBBUTNOTC} and \ref{TypeITypeIBBUTNOTC}). Roughly, this is because the structure of the model function at Denjoy-Wolff points that are not C-points allows one to deduce many different (contractive) versions of Julia's inequality there, thus increasing the supply of invariant horospheres available (see Corollaries  \ref{typeIhorosph} and \ref{TypeIIhorosph}). We also provide examples to illustrate the different cases contained in our theorems. \par 
The paper is arranged as follows. Section \ref{prelims} contains the necessary background on the notions of a model of a function, B-points and C-points and the main result of \cite{HERVE}. It also presents our new definitions of Denjoy-Wolff points and the main results of this paper. In Section \ref{generalBCdir}, we prove general results concerning the relation between the model function and certain directional derivatives at B-points, as well as a refined version of Julia's inequality for the bidisk (see Theorem \ref{Juliarevisited}). These will be much needed in the sequel but are also of independent interest. In Section \ref{charactDW}, we prove several equivalent characterizations of our Denjoy-Wolff points (see Theorems \ref{TypeIcharabridged}, \ref{TypeIIcharabridged} and \ref{ultimateDWchar}), uniqueness results (Propositions \ref{TypeIuniqueness} and \ref{TypeIIuniqueness}) and useful corollaries involving weighted Julia inequalities (Corollaries \ref{typeIhorosph} and \ref{TypeIIhorosph}). Next, in Section \ref{refineHerve}, we revisit Herv\'{e}'s Theorem and establish several partial refinements using our tools from the previous sections. These refinements include Theorems \ref{TypeITypeIIBBUTNOTC}, \ref{TypeITypeIBBUTNOTC} and \ref{motionin(II,II)}. We also provide relevant examples (see Examples \ref{exmp1} and \ref{exmp2}). Finally, in Section \ref{Frosinisection}, we discuss Frosini's work on Denjoy-Wolff-type points on the bidisk and show how our main results can be used to recover a theorem from \cite{FrosiniBusemann} on the classification of a certain type of these points. 

 \small
\section{BACKGROUND AND MAIN RESULTS} \label{prelims}
\large

 \small
\subsection{Models} 
\large 
 Let $\mathcal{S}$ and $\mathcal{S}_2$ denote the one- and two-variable \textit{Schur classes}, i.e. the sets of analytic functions on $\DD$ and $\DD^2$ respectively that are bounded
by $1$ in modulus. We require the notion of a model of a Schur-class function, as seen in \cite{AMYcaratheodory}. It is well known that every function in $\mathcal{S}_2$ possesses such a model, however this ceases to be the case in higher-dimensional polydisks.
\begin{definition} \label{modeldef}
Let $\phi\in\mathcal{S}_2.$ We say that $(M, u)$ is a {\em model} for $\phi$ if $M=M^1\oplus M^2$ is an orthogonally decomposed separable Hilbert space and $u:\DD^2\to M$ is an analytic map
such that, for all $\lambda=(\lambda^1, \lambda^2), \mu=(\mu^1, \mu^2)\in \DD^2,$
\begin{equation}\label{basicmodelform}
1-\phi(\lambda)\overline{\phi(\mu)}=(1-\lambda^1\overline{\mu^1})\langle u^1_{\lambda}, u^1_{\mu}\rangle+(1-\lambda^2\overline{\mu^2})\langle u^2_{\lambda}, u^2_{\mu}\rangle.\end{equation}
\end{definition}
In equation (\ref{modeldef}) we have written $u_{\lambda}$ for $u(\lambda), u^1(\lambda) = P_{M^1} u(\lambda),$ and $u^2(\lambda) = P_{M^2} u(\lambda)$. In general, given $v\in M,$ we will write $v^1$ for $P_{M^1}v$ and $v^2$ for $P_{M^2}v$. Note that we may suppose, without loss of generality, that $\{u^j(\lambda) : \lambda\in\DD^2\}$ spans a dense subspace of $M^j$ , since otherwise we may replace $M^j$ by this span. However, it needn't be true that $\{u(\lambda) : \lambda\in\DD^2\}$ spans a dense subspace of $M$ (these observations can be found in \cite[Section 3]{AMYcaratheodory}).

 \small
\subsection{B-points and C-points} 
\large If $S\subset\DD^2$ and $\tau\in\partial\DD^2$, we say that $S$ \textit{approaches} $\tau$ \textit{nontangentially} if $\tau\in\text{cl}(S)$ (where $\text{cl}(S)$ denotes the topological closure of $S$) and there exists a constant
$c>0$ such that
\begin{equation}\label{nt}
||\tau-\lambda||\le c(1-||\lambda||),
\end{equation}
for all $\lambda\in S,$ where $||(\lambda^1, \lambda^2)||=\max\{|\lambda^1|, |\lambda^2|\}$. \par 
Now, let $\phi\in\mathcal{S}_2$ and $\tau\in\partial\DD^2.$ $\tau$ is said to be a \textit{B-point for} $\phi$ if the Carath\'{e}odory condition
\begin{equation}\label{Caratheodorycond}
\liminf_{\lambda\to\tau} \frac{1-||\phi(\lambda)||}{1-||\lambda||}<\infty  
\end{equation}
 holds. The nontangential limit of $\phi$ at any such $\tau$ always exists \cite{AbateJuliaWolff} and will be denoted by $\phi(\tau).$ \par While in one variable the Julia-Carath\'{e}odory Theorem \cite{Caratheod} tells us that a function in $S$ has an angular derivative at any B-point $\tau,$ a function $\phi\in \mathcal{S}_2$ does not necessarily have an angular gradient at all of its B-points. If $\phi$ does have an angular gradient at $\tau$, we will say that $\tau$ is a \textit{C-point for} $\phi$. In any case, $\phi$ will always have a directional derivative at a $B$-point in any direction pointing into
the bidisk. Moreover, as was shown in \cite{AMYcaratheodory}, the directional derivatives in question will vary holomorphically with respect to direction (actually, the derivatives can be described in terms of certain one-variable \textit{Pick functions} \cite{ATYgeneralizedmodels}, though we won't be needing this result here). \par To state the relevant theorems, we need some notation. Let $(M, u)$ be a model for $\phi\in\mathcal{S}_2$ and define the \textit{nontangential cluster set} $X_{\tau}$ of the model at a B-point $\tau$ of $\phi$ to be the set of weak limits of weakly convergent sequences
$\{u_{\lambda_n}\}$ over all sequences $\{\lambda_n\}$ that converge nontangentially to $\tau$ in $\DD^2$. $X_{\tau}$ turns out to be a subset of the \textit{cluster set} of $(M, u)$ at $\tau$, which is defined as the set of limits in $M$ of the weakly convergent sequences $\{u_{\lambda_n}\}$ as $\{\lambda_n\}$ ranges over all sequences in $\DD^2$ that tend to $\tau$ in such a way that 
\begin{equation}\label{clustersetcondition}
  \frac{1-|\phi(\lambda_n)|}{1-||\lambda_n||}  
\end{equation}
remains bounded. The cluster set at $\tau$ will be denoted by $Y_{\tau}.$ Also, let $\mathbb{H}=\{z\in\mathbb{C} : \Re z>0\}, \mathbb{T}=\partial\DD$ and define, for every $\tau\in\partial\mathbb{D}^2,$
$$\mathbb{H}(\tau)=\begin{cases} \tau^1\mathbb{H}\times\tau^2 \mathbb{H} \hspace{0.6 cm} \text{if }\tau\in\mathbb{T}^2, \\
\tau^1\mathbb{H}\times\mathbb{C} \hspace{0.6 cm} \text{if }\tau\in\mathbb{T}\times\DD, \\
\mathbb{C}\times\tau^2 \mathbb{H} \hspace{0.6 cm} \text{if }\tau\in\DD\times\mathbb{T}.
\end{cases}
$$
\par For the remainder of this subsection, fix a function $\phi\in \mathcal{S}_2$ with model $(M, u)$ and a B-point $\tau\in\partial\DD^2.$ The next lemma can be easily obtained from (\ref{basicmodelform}).
\begin{lemma}[see \cite{AMYcaratheodory}, Proposition 4.2]\label{basicmodelsetup}
We have $X_{\tau}\neq \emptyset$. Moreover, for all $x\in Y_{\tau}$ and $\lambda\in\DD^2$,
\begin{equation} \label{basicmodelsetupeq}
 1-\phi(\lambda)\overline{\phi(\tau)}=\sum_{|\tau^j|=1}(1-\lambda^j\overline{\tau^j})\langle u^j_{\lambda}, x^j\rangle.
\end{equation}
\end{lemma}
As a consequence, we obtain:
\begin{lemma}[see \cite{AMYcaratheodory}, Lemma 8.10] \label{facialBisC}
 If $|\tau^j|<1$ for $j=1$ or $2,$ then $$Y_{\tau}=\{u_{\tau}\}, \hspace{0.2 cm} \text{ where }u^j_{\tau}=0.$$
\end{lemma}
A consequence of the following theorem is that facial B-points are always C-points (see \cite{AMYfacial} for more results in that direction).
\begin{theorem}[see \cite{AMYcaratheodory}, Corollary 8.11] \label{facialB}
$\tau$ is a C-point for $\phi$ if and only if $X_{\tau}$ is a singleton set.
\end{theorem}
Now, since $\tau$ is a B-point for $\phi,$ we know that for every $\delta\in\mathbb{H}(\tau)$ the directional derivative
$$ D_{-\delta}\phi(\tau)=\lim_{t\to 0+}\frac{\phi(\tau-t\delta)-\phi(\tau)}{t}$$
exists. Much more can be said.
\begin{theorem}[see \cite{AMYcaratheodory}, Theorems 7.1, 7.8]\label{generaldir}
For any $\delta\in\mathbb{H}(\tau)$, the nontangential limit (in the norm of $M$)
$$x_{\tau}(\delta)=\lim_{\tau-z\delta\xrightarrow{nt}\tau}u_{\tau-z\delta}$$
exists in $M$. In addition,
\begin{itemize}
    \item[(1)] $x_{\tau}(\cdot)$ is a holomorphic $M$-valued function on $\mathbb{H}(\tau)$;
    \item[(2)] $x_{\tau}(\delta)\in X_{\tau}$ for all $\delta\in\mathbb{H}(\tau)$;
    \item[(3)]  $x_{\tau}(z\delta)=x_{\tau}(\delta)$ for all $z\in\mathbb{C}$ such that $\delta, z\delta\in\mathbb{H}(\tau)$ (i.e. $x_{\tau}(\cdot)$ is homogeneous of degree $0$ in $\delta$);
    \item[(4)] $D_{-\delta}\phi(\tau)$ is analytic, homogeneous of degree $1$ in $\delta$ and satisfies 
    $$D_{-\delta}\phi(\tau)=-\phi(\tau)\sum_{|\tau^j|=1}\overline{\tau^j}\delta^j||x^j_{\tau}(\delta)||^2.$$
\end{itemize} 
\end{theorem}

 \small
\subsection{Horocycles and horospheres} 
\large 
The language of horospheres and horocycles will be required for our iteration-theoretic results. Recall that a \textit{horocycle} in $\DD$ is a set of the form $E(\tau, R)$ for some $\tau\in\text{cl}(\DD)$ and $R>0,$ where 
$$E(\tau, R)=\bigg\{\lambda\in\DD : \frac{|\lambda-\tau|^2}{1-|\lambda|^2}<R \bigg\}$$
for $\tau\in\mathbb{T},$ while $E(\tau, R)=\DD$ otherwise. Letting $D(z, r)$ denote the Euclidean disk in $\mathbb{C}$ with centre $z$ and radius $r>0,$ it is not hard to see that, given any $\tau\in\mathbb{T}$, we always have
$$E(\tau, R)=D\bigg(\frac{\tau}{R+1}, \frac{R}{R+1}\bigg).$$
Also, for $\tau=(\tau^1, \tau^2)\in\partial\DD^2$ and $R_1, R_2>0,$ we define the (weighted) \textit{horosphere} $E(\tau, R_1, R_2)$ to be the set $E(\tau^1, R_1)\times E(\tau^2, R_2)$. \par 
Now, given $\phi\in\mathcal{S}$ and a B-point $\tau\in\mathbb{T},$ it is known that 
$$\alpha:=\lim_{\lambda\xrightarrow{\text{nt}}\tau}\frac{1-|\phi(\lambda)|}{1-|\lambda|}\ge 0$$
exists. Julia's inequality \cite{Caratheod}, \cite{Julia} then states that 
\begin{equation}\label{Juliaineq}
 \phi\big(E(\tau, R) \big)\subset E(\phi(\tau), \alpha R),  
\end{equation}
for all $R>0.$ Generalizations of this result to the bidisk are contained in  \cite{Abatebook} and \cite{Wlod} (see also \cite[Section 4]{AMYcaratheodory} for a model-theoretic proof). In particular, given $\phi\in\mathcal{S}_2$ and a B-point $\tau\in\partial\DD^2$, it is known that, for any $\alpha\ge 0$, we have
$$\liminf_{\lambda\to\tau}\frac{1-|\phi(\lambda)|}{1-||\lambda||}\le \alpha $$
if and only if 
\begin{equation}\label{classicalJulia}
  \phi(E(\tau, R, R))\subset E(\phi(\tau), \alpha R),  
\end{equation}
for all $R>0$ (if $\alpha=0,$ then $\phi$ is constant). In Section \ref{generalBCdir}, we use ideas from \cite{AMYcaratheodory} to establish a refined version of the previous equivalence, one that is expressed in terms of weighted horospheres (see Theorem \ref{Juliarevisited}). \par 
Lastly, we will occasionally be making use of the \textit{horospheric topology} on $\text{cl}(\DD^2)$, which is the topology with base consisting of all open sets of $\DD^2$ together with all sets of the form $\{\tau\}\cup E(\tau, R_1, R_2),$ where $\tau\in\partial\DD^2$ and $R_1, R_2>0$ (see \cite[Section 4]{AMYcaratheodory} for more details). Note that (\ref{classicalJulia}) tells us that $\phi(\lambda)\to \phi(\tau)$ whenever $\tau$ is a B-point and $\lambda\to\tau$ horospherically.

\small
\subsection{Herv\'{e}'s result} 
\large 
 For $i\in\{1, 2\}$, define the coordinate projections $\pi^i:\DD^2\to\DD, \pi^i(\lambda)=\lambda^i.$ Given $\phi\in\mathcal{S}_2$ and $\mu\in\DD$, we will denote by $\phi_{\mu}\in\mathcal{S}$ the slice function $$\phi_{\mu}(\lambda)=\phi(\lambda, \mu) \hspace{0.4 cm} (\lambda\in\DD).$$ Also, we let $\widetilde{\phi}\in\mathcal{S}_2$ denote the function
 $\widetilde{\phi}(\lambda)=\phi(\lambda^2, \lambda^1), \text{ for }$ obtained from $\phi$ by interchanging the arguments. \par 
 Holomorphic functions $\phi: \DD^2\to\DD$ can be classified according to the Denjoy-Wolff points of their slices. 

\begin{definition}\label{Type I, Type II} 
Assume $\phi\in\mathcal{S}_2$. $\phi$ is said to be a:  \begin{itemize}
    \item[(i)] \textit{left Type I} function if $\phi\neq\pi^1$ and there exists $\tau^1\in\mathbb{T}$ such that $\tau^1$ is the common Denjoy-Wolff point of the maps $\phi_{\mu}\in\mathcal{S},$ for all $\mu\in\DD$;
    \item[(ii)]\textit{right Type I} function if $\widetilde{\phi}$ is a left Type I function;
    \item[(iii)]\textit{left Type II} function if $\phi\neq\pi^1$ and there exists a holomorphic map $\xi:\DD\to\DD$ such that, for all $\lambda, \mu\in\DD$, we have $\phi_{\mu}(\lambda)=\lambda$ if and only if $\xi(\mu)=\lambda$;
    \item[(iv)]\textit{right Type II} function if $\widetilde{\phi}$ is a left Type II function.
\end{itemize}
\end{definition}
Surprisingly, it turns out that any $\phi\in\mathcal{S}_2$ that is not a coordinate projection will either be a left Type I or a left Type II function (respectively, either a right Type I or a right Type II function), a result originally proved by Herv\'{e} in \cite{HERVE}. In Section \ref{charactDW}, we give a new proof of this using purely model-theoretic methods (see Theorem \ref{130}). 
\par 
Using the Type I/Type II terminology, the main result of \cite{HERVE} can be stated as follows. 
\begin{theorem}[Herv\'{e}] \label{HERVE's THEOREM}
    Let $F=(\phi, \psi):\DD^2\to\DD^2$ be a holomorphic self-map of the bidisk without fixed points.  Then, one and only one of the following cases
occurs: 
\begin{itemize}
    \item[(i)] if $\psi\equiv\pi^2$ (respectively, $\phi\equiv\pi^1$), then $\{F^n\}$ converges uniformly on
compact sets to $(\tau^1, \pi^2),$ where $\tau^1\in\mathbb{T}$ (respectively, to $(\pi^1, \tau^2)$, where $\tau^2\in\mathbb{T}$);
\item[(ii)] if $\phi$ is a left Type I and $\psi$ is right Type I function, then there exist $\tau^1, \tau^2\in\mathbb{T}$ such that \begin{itemize}
    \item[(a)] either every cluster point of $\{F^n\}$ has the form $(\tau^1, h),$ where $h$ is either a holomorphic function $\DD^2\to\DD$ or the constant $\tau^2$,
    \item[(b)] or every cluster point of $\{F^n\}$ has the form $(g, \tau^2),$ where $g$ is either a holomorphic function $\DD^2\to\DD$ or the constant $\tau^1$;
\end{itemize} 
\item[(iii)] if $\phi$ is a left Type I function and $\psi$ is a right Type II function (respectively, $\phi$ is a left Type II function and $\psi$ is a right Type I function), there exists $\tau^1\in\mathbb{T}$ such that every cluster point of $\{F^n\}$ has the form $(\tau^1, h),$ where $h\in\mathcal{S}_2$ (respectively, there exists $\tau^2\in\mathbb{T}$ such that every cluster point of $\{F^n\}$ has the form $(g, \tau^2),$ where $g\in\mathcal{S}_2$);
\item[(iv)] if $\phi$ is a left Type II and $\psi$ is a right Type II function, then there exist $\tau^1, \tau^2\in\mathbb{T}$ such that $\{F_n\}$ converges uniformly on compact sets to $(\tau^1, \tau^2).$
\end{itemize}
\end{theorem}

 \small
\subsection{Principal results} 
\large 
We begin with our model-theoretic definitions of Denjoy-Wolff-type points. 
\begin{definition}\label{TypeI+IIdef}   Let $\phi\in\mathcal{S}_2$ with model $(M, u).$  Assume first that $\phi\neq \pi^1$.
\begin{itemize}
    \item[(i)] A point $(\tau^1, \sigma)\in\mathbb{T}\times\text{cl}(\DD)$ will be called a \textit{left Type I DW point} for $\phi$ if it is a B-point, $\phi(\tau^1, \sigma)=\tau^1$  and there exists $u_{(\tau^1, \sigma)}\in Y_{(\tau^1, \sigma)}$ such that $||u^1_{(\tau^1, \sigma)}||\le 1$ and $u^2_{(\tau^1, \sigma)}=0$.    
     \item[(ii)] A point $\tau=(\tau^1, \tau^2)\in\mathbb{T}^2$ will be called a \textit{left Type II DW point} for $\phi$ if it is a B-point, $\phi(\tau)=\tau^1$, there exists $u_{\tau}\in Y_{\tau}$ such that $||u^1_{\tau}||<1$ and $\tau$ is not a left Type I DW point for $\phi.$ In particular, if $K>0$ is any constant such that 
     $$||u^1_{\tau}||^2+K||u^2_{\tau}||^2\le 1,$$
     we will say that $\tau$ is a \textit{left Type II DW point with constant $K$}.
      \end{itemize} 
Now, assume instead that $\phi\neq \pi^2$.
     \begin{itemize}
       \item[(iii)]  A point $(\sigma, \tau^2)\in\text{cl}(\DD)\times\mathbb{T}$ will be called a \textit{right Type I DW point} for $\phi$ if $(\tau^2, \sigma)$ is a left Type I DW point for $\widetilde{\phi}.$
     \item[(iv)] A point $\tau=(\tau^1, \tau^2)\in\mathbb{T}^2$ will be called a \textit{right Type II DW point} for $\phi$ (with constant $K>0$) if $\widetilde{\tau}=(\tau^2, \tau^1)$ is a left Type II DW point for $\widetilde{\phi}$ (with constant $K>0$).
 \end{itemize}
\end{definition}

 An immediate consequence of Definition \ref{TypeI+IIdef} is that every left (resp., right) Type II DW point is a left (resp., right) Type II DW point with constant $K$, for some $K>0.$  \par 
 The following characterizations are proved in Section \ref{charactDW} (notice that the property of being a Type I/Type II point turns out not to depend on the model of the function). 
 \begin{theorem}\label{TypeIcharabridged}
  Let $\phi\in\mathcal{S}_2$ with model $(M, u)$ and $\tau^1\in\mathbb{T}.$  Assume also that $\phi\neq \pi^1$. The following assertions are equivalent: 
  \begin{itemize}
      \item[(i)] there exists $\sigma\in\text{cl}(\DD)$ such that $(\tau^1, \sigma)$ is a left Type I DW point for $\phi;$
      \item[(ii)] every point in $\{\tau^1\}\times\text{cl}(\DD)$ is a left Type I DW point for $\phi$;      
      \item[(iii)] $\phi$ is a left Type I function and the common Denjoy-Wolff point of all slice functions $\phi_{\mu}\in\mathcal{S}$ is $\tau^1$;
       \item[(iv)] there exists $\sigma\in\text{cl}(\DD)$ such that $(\tau^1, \sigma)$ is a B-point, $\phi(\tau^1, \sigma)=\tau^1$  and $$\frac{D_{-(\tau^1, \sigma M)}\phi(\tau^1, \sigma)}{-\tau^1}\le 1, \hspace{0.4 cm} \forall M>0;$$
      \item[(v)] for every $\sigma\in\text{cl}(\DD)$, $(\tau^1, \sigma)$ is a B-point, $\phi(\tau^1, \sigma)=\tau^1$  and  $$\frac{D_{-(\tau^1, \sigma M)}\phi(\tau^1, \sigma)}{-\tau^1}\le 1, \hspace{0.4 cm} \forall M>0.$$
       \end{itemize} \par
Moreover, assuming that any of the above statements holds and letting $\phi'_{\mu}(\tau^1)$ denote the angular derivative of $\phi_{\mu}$ at $\tau^1$, we obtain 
$$\lim_{M\to\infty}D_{-(\tau^1, \sigma M)}\phi(\tau^1, \sigma)=-\tau^1\phi'_{\mu}(\tau^1),$$
for all $\mu\in\DD$ and all $|\sigma|\le 1.$ \par 
There is an analogous statement for right Type I DW points (we need to assume that $\phi\neq \pi^2$).

 \end{theorem}
    
\begin{theorem} \label{TypeIIcharabridged}
 Let $\phi:\DD^2\to\DD$ be holomorphic with model $(M, u)$. Also, let $\tau=(\tau^1, \tau^2) \in\mathbb{T}^2$, $K>0$  and assume that $\phi\neq \pi^1.$ The following assertions are equivalent: 
 \begin{itemize}
 \item[(i)] $\tau$ is a left Type II DW point for $\phi$ with constant $K$;
 \item[(ii)] $\phi$ is a left Type II function. Also, letting $\xi:\DD\to\DD$ denote the holomorphic function such that $\phi(\xi(\mu), \mu)=\xi(\mu),$ for all $\mu\in\DD,$ we have that $\tau^2$ is a B-point for $\xi,$ $\xi(\tau^2)=\tau^1$ and 
 $$\liminf_{z\to\tau^2}\frac{1-|\xi(z)|}{1-|z|}\le \frac{1}{K};$$
 \item[(iii)] $\tau$ is a B-point for $\phi,$ $\phi(\tau)=\tau^1$, the quantity $D_{-(\tau^1, \tau^2 M)}\phi(\tau)$ is not constant with respect to $M>0$ and also there exists $A\ge K$ such that 
 $$\frac{D_{-(\tau^1, \tau^2 A)}\phi(\tau)}{-\tau^1}=1.$$

\end{itemize}
Moreover, assuming that any of the above statements holds, $$A=\bigg[\liminf_{z\to\tau^2}\frac{1-|\xi(z)|}{1-|z|}\bigg]^{-1}$$ will be the maximum among all constants $K>0$ such that $\tau$ is a left Type II DW point for $\phi$ with constant $K$. It will also be the unique positive number such that $D_{-(\tau^1, \tau^2 A)}\phi(\tau)/(-\tau^1)=1.$
\par 
    There is an analogous statement for right Type II DW points (we need to assume that $\phi\neq \pi^2$).
\end{theorem}
A consequence of Theorem \ref{TypeIIcharabridged} is that not all Type II functions have Type II DW points (just choose e.g. any left Type II function such that $\xi$ has no B-points). However, Type II DW points do appear naturally when investigating iteration-theoretic questions. In particular, if $F=(\phi, \psi):\DD^2\to\DD^2$ has no fixed points, $\phi$ is left Type II and $\psi$ is right Type II, then both $\phi$ and $\psi$ will have Type II DW points (see Theorem \ref{(Type II, Type II)} for details). 
\par 
Theorems \ref{TypeIcharabridged}-\ref{TypeIIcharabridged} allow us to give a simple, unified characterization of Type I/II DW points, one that is expressed in terms of directional derivatives and is easier to verify in practice than checking for invariant horospheres.  
To state it, set (for any function $\phi\in\mathcal{S}_2$ such that $\tau\in\partial\DD^2$ is a B-point)
$$K_{\tau}(M)=\frac{D_{-(\tau^1, \tau^2 M)}\phi(\tau)}{-\phi(\tau)} \hspace{0.4 cm} (M>0).$$
It can be shown (see Proposition \ref{directdersincreasing}) that $K_{\tau}(M)$ is nonnegative and increasing with respect to $M$. This observation, combined with Theorems \ref{TypeIcharabridged}-\ref{TypeIIcharabridged}, leads to:
\begin{theorem}\label{ultimateDWchar}
Let $\phi\in\mathcal{S}_2$ and assume $\tau=(\tau^1, \tau^2)\in\partial\DD^2$ is a B-point for $\phi$ such that $\phi(\tau)=\tau^1.$ Assume also that $\phi\neq\pi^1.$
\begin{itemize}
    \item[(a)] If $|\tau^2|<1,$ then $\tau$ is a left Type I DW point that is also a C-point for $\phi$ if and only if $$K_{\tau}(M)=\alpha\le 1, \hspace{0.4 cm} \forall M>0.$$  In any other case, $\tau$ will be neither a left Type I nor a left Type II DW point.
    \item[(b)] If $|\tau^2|=1,$ then $\tau$ is a:  
    
    \begin{itemize}
        \item[(i)] left Type I DW point that is also a C-point if and only if $$K_{\tau}(M)=\alpha\le 1, \hspace{0.4 cm} \forall M>0;$$ 
        \item[(ii)] left Type I DW point that is not a C-point if and only if  $\{K_{\tau}(M)\}_M$ is non-constant and $$K_{\tau}(M)<1, \hspace{0.4 cm} \forall M>0;$$
        \item[(iii)] left Type II DW point if and only if  $\{K_{\tau}(M)\}_M$ is non-constant and there exists $A>0$ such that $$K_{\tau}(A)=1;$$ 
        \item[(iv)] neither a left Type I nor a left Type II DW point if and only if $$K_{\tau}(M)>1, \hspace{0.4 cm} \forall M>0.$$

    \end{itemize}
\end{itemize}
  There is an analogous statement for right Type I/II DW points (we need to assume that $\phi\neq \pi^2$).  
\end{theorem}
Using our work on DW points, we are able to offer the following refinements of Theorem \ref{HERVE's THEOREM}. 
\begin{theorem}\label{TypeITypeIIBBUTNOTC}
Assume $F=(\phi, \psi):\DD^2\to\DD^2$ is holomorphic, $\phi$ is left Type I and $\psi$ is right Type II. Let $\tau^1$ denote the common Denjoy-Wolff point of all slice functions $\phi_{\mu}.$ If there exists $\sigma\in\mathbb{T}$ such that $(\tau^1, \sigma)$ is a right Type II DW point for $\psi$ but not a C-point for $\phi$, then $F^n\to(\tau^1, \sigma)$ uniformly on compact subsets of $\DD^2.$
\end{theorem}
\begin{theorem}\label{TypeITypeIBBUTNOTC}
Assume $F=(\phi, \psi):\DD^2\to\DD^2$ is holomorphic, $\phi$ is left Type I and $\psi$ is right Type I. Let $\tau^1$ and $\tau^2$ denote the common Denjoy-Wolff points of all slice functions $\phi(\cdot, \mu)$ and $\psi(\lambda, \cdot)$, respectively. If $\tau=(\tau^1, \tau^2)$ is not a C-point for $\phi$, then every cluster point of $\{F^n\}$ will have the form $(\tau^1, h),$ where $h$ is either a holomorphic function $\DD^2\to\DD$ or the constant $\tau^2$. 
An analogous conclusion can be reached if $\tau$ is not a C-point for $\psi.$
\end{theorem}
Applications are contained in Examples \ref{exmp1} and \ref{exmp2}. A further refinement can be found in Theorem \ref{motionin(II,II)}.

 \small
\section{B-POINTS AND DIRECTIONAL DERIVATIVES ALONG $(\tau^1, \tau^2 M)$} \label{generalBCdir}
\large This section contains several technical results that build upon the model theory of \cite{AMYcaratheodory} and \cite{AMYfacial}, the highlights being Theorems \ref{nullcompgivesCpoint} and \ref{BbutnotC}-\ref{Juliarevisited}. These will be critical for our work in Sections \ref{charactDW}, \ref{refineHerve}, but are also interesting in their own right.  \par

Now, choose an arbitrary $\phi\in\mathcal{S}_2$ with model $(M, u)$ and a B-point $\tau=(\tau^1, \tau^2)\in\partial\DD^2$. These will be fixed for the remainder of this section. Recall that we can define  
$$x_{\tau}(\delta)=\lim_{\tau-z\delta\xrightarrow{nt}\tau}u_{\tau-z\delta},$$
for any $\delta\in\mathbb{H}(\tau)$, where the limit is with respect to the norm of $M$. The following easy consequence of (\ref{basicmodelsetupeq}) will be used repeatedly throughout the paper. 
\begin{lemma}\label{comparisonsetup}
Assume $\tau\in\mathbb{T}^2$. Then, for any $u_{\tau}\in Y_{\tau}$ we have
$$\langle x^1_{\tau}(\delta), u^1_{\tau}\rangle+\frac{\overline{\tau^2}\delta^2}{\overline{\tau^1}\delta^1}\langle  x^2_{\tau}(\delta), u^2_{\tau}\rangle=||x^1_{\tau}(\delta)||^2+\frac{\overline{\tau^2}\delta^2}{\overline{\tau^1}\delta^1}||x^2_{\tau}(\delta)||^2,$$
for all $\delta\in\mathbb{H}(\tau).$
\end{lemma}
\begin{proof}
Applying (\ref{basicmodelsetupeq}) twice gives us
$$\langle u^1_{\lambda}, u^1_{\tau}\rangle+\frac{1-\lambda^2\overline{\tau^2}}{1-\lambda^1\overline{\tau^1}}\langle u^2_{\lambda}, u^2_{\tau}\rangle=\langle u^1_{\lambda}, x^1_{\tau}(\delta)\rangle+\frac{1-\lambda^2\overline{\tau^2}}{1-\lambda^1\overline{\tau^1}}\langle u^2_{\lambda}, x^2_{\tau}(\delta)\rangle,$$
for all $\lambda\in\DD^2$ and $\delta\in\mathbb{H}(\tau).$ Setting $\lambda=\tau-r\delta$ and letting $r\to 0+$ then finishes off the proof.
\end{proof}
We also require the following lemma.
\begin{lemma} \label{twonullcoords}
   Assume $\tau\in\mathbb{T}^2$. If $u_{\tau}, v_{\tau}\in Y_{\tau}$ are such that $u^i_{\tau}=v^i_{\tau}=0$ for some $i\in\{1, 2\},$ we must have $u_{\tau}=v_{\tau}.$
\end{lemma}
\begin{proof}
    Without loss of generality, assume $i=2.$ Applying (\ref{basicmodelsetupeq}) twice, we obtain
$$ 1-\phi(\lambda)\overline{\phi(\tau)}=(1-\lambda^1\overline{\tau^1})\langle u^1_{\lambda}, u^1_{\tau}\rangle, $$
$$=(1-\lambda^1\overline{\tau^1})\langle u^1_{\lambda}, v^1_{\tau}\rangle,$$
for all $\lambda\in\DD^2.$ Thus, $\langle u^1_{\lambda}, u^1_{\tau}-v^1_{\tau}\rangle=0$ for all $\lambda.$ This equality, combined with the fact that both $v^1_{\tau}$ and $v^1_{\tau}$ are weak limits of vectors in the span of $\{u^1_{\lambda} : \lambda\in\DD^2\}$ implies that 
$$||u^1_{\tau}||^2=||v^1_{\tau}||^2=\langle u^1_{\tau}, v^1_{\tau}\rangle.$$
Thus, $v^1_{\tau}=u^1_{\tau}$ and we are done.
\end{proof}
Our next result shows that the presence of vectors with null components in $X_{\tau}$ has a surprisingly strong impact on the boundary regularity of the function. We exclude facial B-points from our theorem, since they are automatically C-points.
\begin{theorem}\label{nullcompgivesCpoint}
 Assume $\tau\in\mathbb{T}^2$ and also that there exists $x_{\tau}(\delta)\in X_{\tau}$ with $x^i_{\tau}(\delta)=0$ for some $i\in\{1, 2\}.$ Then,
 $\tau$ is a C-point for $\phi.$  
\end{theorem}
\begin{proof}
Without loss of generality, assume that there exists  $x_{\tau}(\delta_0)\in X_{\tau}$ with $x^2_{\tau}(\delta_0)=0$. We may assume that $x^1_{\tau}(\delta_0)\neq 0$, else $\phi$ would be a unimodular constant. In view of Lemma \ref{comparisonsetup}, we obtain 
\begin{equation}\label{someq}
 \langle x^1_{\tau}(\delta), x^1_{\tau}(\delta_0)\rangle=||x^1_{\tau}(\delta)||^2+\frac{\overline{\tau^2}\delta^2}{\overline{\tau^1}\delta^1}||x^2_{\tau}(\delta)||^2,   
\end{equation}
for all $\delta\in\mathbb{H}(\tau).$ Choose any open subset $\Omega$ of $\mathbb{H}(\tau)$ with the property that $\frac{\overline{\tau^2}\delta^2}{\overline{\tau^1}\delta^1}$ has positive real part for all $\delta\in\Omega.$ (\ref{someq}) then implies that $$||x^1_{\tau}(\delta)||\le ||x^1_{\tau}(\delta_0)||$$ for all $\delta\in\Omega.$ Indeed, if this were not the case, we would be able to write
$$\Re \langle x^1_{\tau}(\delta), x^1_{\tau}(\delta_0)\rangle\le || x^1_{\tau}(\delta)||\cdot ||x^1_{\tau}(\delta_0)||  $$
$$< ||x^1_{\tau}(\delta)||^2 $$
$$\le ||x^1_{\tau}(\delta)||^2 +\Re\bigg(\frac{\overline{\tau^2}\delta^2}{\overline{\tau^1}\delta^1}\bigg)||x^2_{\tau}(\delta)||^2 $$
whenever $\delta\in\Omega,$ a contradiction. \par 
Now, assume $||x^1_{\tau}(\delta)||=||x^1_{\tau}(\delta_0)||$ for all $\delta\in\Omega.$ The previous chain of inequalities then implies that
$$\langle x^1_{\tau}(\delta), x^1_{\tau}(\delta_0)\rangle=||x^1_{\tau}(\delta)||^2=|| x^1_{\tau}(\delta_0) ||^2,$$
for all $\delta\in\Omega.$ This gives us $x^1_{\tau}(\delta)=x^1_{\tau}(\delta_0)$ on $\Omega$, and hence also on $\mathbb{H}(\tau)$. In view of (\ref{someq}), we obtain that $x^2_{\tau}(\cdot)$ must be identically zero. Hence, $X_{\tau}=\{(x^1_{\tau}(\delta_0), 0)\}$ and we obtain (by Lemma \ref{facialB}) that $\tau$ is a C-point. \par 
Assume, on the other hand, that we can find $\delta_1\in\mathbb{H}(\tau)$ such that $||x^1_{\tau}(\delta_1)||<||x^1_{\tau}(\delta_0)||$. Applying \ref{basicmodelsetupeq} again, with $\delta=\delta_0$ and $u_{\tau}=x_{\tau}(\delta_1)$, we deduce that 
$$ \langle x^1_{\tau}(\delta_0),x^1_{\tau}(\delta_1)\rangle=||x^1_{\tau}(\delta_0)||^2,$$
a contradiction. This concludes the proof.
\end{proof}

\begin{remark}
If we merely assume the existence of $u_{\tau}\in Y_{\tau}$ such that $u^i_{\tau}=0$ for some $i\in\{1, 2\}$, $\tau$ will not necessarily be a C-point; see Example \ref{exmp1}.
\end{remark}

Next, we show that the directional derivatives of $\phi$  along $(\tau^1, \tau^2M)$ can be naturally associated with an increasing (with respect to $M$) sequence  of positive numbers. Indeed, put $\delta_M=(\tau^1, \tau^2M)$ and define 
$$K_{\tau}(M):=\frac{D_{-\delta_M}\phi(\tau)}{-\phi(\tau)}=||x^1_{\tau}(\delta_M)||^2+M||x^2_{\tau}(\delta_M)||^2,$$
for all $M>0.$
\begin{proposition} \label{directdersincreasing}
For any $u_{\tau}\in Y_{\tau}$ we have 
$$K_{\tau}(M)\le ||u^1_{\tau}||^2+M||u^2_{\tau}||^2, \hspace{0.4 cm} \forall M>0,$$
with equality if and only if $x_{\tau}(\delta_M)=u_{\tau}$.
In particular, $K_{\tau}(M)$ is increasing with respect to $M.$ It will be strictly increasing if and only if $X_{\tau}\neq\{(x^1_{\tau}, 0)\}$.
\end{proposition} 
\begin{proof}
\par First, assume $\tau$ is a facial B-point. If $|\tau^2|<1,$ then Lemma \ref{facialBisC} tells us that $Y_{\tau}=X_{\tau}=\{(x^1_{\tau}, 0)\},$ hence $K_{\tau}(M)$ is constant and there is nothing to prove. If $|\tau^1|<1$, then $Y_{\tau}=X_{\tau}=\{(0, x^2_{\tau})\}$, $K_{\tau}(M)$ is strictly increasing with respect to $M$ and the theorem obviously holds. \par 
 Now, assume $\tau\in\mathbb{T}^2$ and fix $u_{\tau}\in Y_{\tau}$, $M>0.$ In view of Lemma \ref{comparisonsetup}, we can apply Cauchy-Schwarz to obtain
 $$K_{\tau}(M)=\langle x^1_{\tau}(\delta_M), u^1_{\tau}\rangle +M\langle x^2_{\tau}(\delta_M), u^2_{\tau}\rangle $$
 \begin{equation} \label{yeah1}
  \le ||x^1_{\tau}(\delta_M)||\cdot||u^1_{\tau}||+\big(\sqrt{M}||x^2_{\tau}(\delta_M)||\big)\big(\sqrt{M}||u^2_{\tau}||\big)   
 \end{equation} 
 \begin{equation} \label{yeah2}
 \le \sqrt{K_{\tau}(M)}\sqrt{||u^1_{\tau}||^2+M||u^2_{\tau}||^2}.    
 \end{equation}
 Thus, $K_{\tau}(M)\le ||u^1_{\tau}||^2+M||u^2_{\tau}||^2.$ \par 
 When does equality hold? For (\ref{yeah1}) to hold as an equality, we must have $c^i\in\mathbb{R}^{+}\cup\{0\}$ such that $c^ix^{i}_{\tau}(\delta_M)=u^i_{\tau}$, for $i\in\{1, 2\}.$ For (\ref{yeah2}), we need
 \begin{equation} \label{CSeq}
 {||x^1_{\tau}(\delta_M)||}\cdot{||u^2_{\tau}||}={||x^2_{\tau}(\delta_M)||}\cdot{||u^1_{\tau}||}.\end{equation} \par 
 Now, assume that either $x^i_{\tau}(\delta_M)=0$ or $u^i_{\tau}=0$ for some $i.$ For definiteness, let us assume $u^2_{\tau}=0$ (the other cases are proved in an identical way). In view of (\ref{CSeq}), we must have either $x^2_{\tau}(\delta_M)=0$ or $u^1_{\tau}=0.$ If the latter holds, we obtain $u_{\tau}=0,$ hence $\phi$ is a unimodular constant and there is nothing to prove. Thus, we may assume $x^2_{\tau}(\delta_M)=0$. In this case, we may replace $u^1_{\tau}$ by $c^1x^{1}_{\tau}(\delta)$ in the equality $$||x^1_{\tau}(\delta_M)||^2=K_{\tau}(M)=\langle x^1_{\tau}(\delta_M), u^1_{\tau}\rangle +M\langle x^2_{\tau}(\delta_M), u^2_{\tau}\rangle=\langle x^1_{\tau}(\delta_M), u^1_{\tau}\rangle$$
 to obtain $||x^1_{\tau}(\delta_M)||^2=c^1||x^1_{\tau}(\delta_M)||^2.$ If $x^1_{\tau}(\delta_M)=0,$ we again obtain that $\phi$ is a unimodular constant, while $x^1_{\tau}(\delta_M)\neq 0$ implies $c^1=1,$ hence $x_{\tau}(\delta_M)=u_{\tau}$. \par 
 On the other hand, assume $x^1_{\tau}(\delta_M), x^2_{\tau}(\delta_M), u^1_{\tau}, u^2_{\tau}$ are all nonzero. (\ref{CSeq}) then gives us $c^1=c^2=c.$ Replacing $u^i_{\tau}$ by $cx^{i}_{\tau}(\delta_M)$ in the equality $$K_{\tau}(M)=\langle x^1_{\tau}(\delta_M), u^1_{\tau}\rangle +M\langle x^2_{\tau}(\delta_M), u^2_{\tau}\rangle,$$ we obtain $c=1,$ hence $x_{\tau}(\delta_M)=u_{\tau}$. \par 
 Now, we show that $K_{\tau}(M)$ is increasing with respect to $M.$ Indeed, let $N>M>0.$ Setting $u_{\tau}=x_{\tau}(\delta_N)$ in our previous result implies
 $$K_{\tau}(M)\le ||x^1_{\tau}(\delta_N)||+M||x^2_{\tau}(\delta_N)||\le K_{\tau}(N),$$ as desired. \par 
 Now, if $X_{\tau}=\{(x^1_{\tau}, 0)\}$, it is evident that  $K_{\tau}(M)$ will be constant (and equal to $||x^1_{\tau}||^2$ for all $M$). On the other hand, assume $X_{\tau}$ is not a singleton of the form $\{(x^1_{\tau}, 0)\}$ but that we can also find positive numbers $M<N$ such that $K_{\tau}(M)=K_{\tau}(N)$. As we have already seen, this implies that $x_{\tau}(\delta_M)=x_{\tau}(\delta_N),$ which, combined with $K_{\tau}(M)=K_{\tau}(N)$, allows us to deduce that $x^2_{\tau}(\delta_M)=x^2_{\tau}(\delta_N)=0.$  Theorem \ref{nullcompgivesCpoint} then tells us that $X_{\tau}=\big\{\big(x^1_{\tau}(\delta_M), 0\big)\big\}$, a contradiction.
\end{proof}

We now explore some consequences of Proposition \ref{directdersincreasing}.

\begin{corollary}\label{somecor}
Given any $M>0$ and $u_{\tau}\in Y_{\tau},$ we must either have $||x^1_{\tau}(\delta_M)||\le ||u^1_{\tau}||$ or $||x^2_{\tau}(\delta_M)||\le ||u^2_{\tau}||$. Moreover, if 
\begin{itemize}
    \item[(i)] $||x^1_{\tau}(\delta_M)||=||u^1_{\tau}||$ (resp., $||x^2_{\tau}(\delta_M)||=||u^2_{\tau}||$), then $||x^2_{\tau}(\delta_M)||\le||u^2_{\tau}||$ (resp., $||x^1_{\tau}(\delta_M)||\le||u^1_{\tau}||$);
\item[(ii)]  $||x_{\tau}(\delta_M)||=||u_{\tau}||,$ then
$x_{\tau}(\delta_M)=u_{\tau}$.
\end{itemize}
\end{corollary}
\begin{proof}
Assume that $||x^i_{\tau}(\delta_M)||>|u^i_{\tau}||$ for all $i\in\{1, 2\}.$ Thus, $K_{\tau}(M)>||u^1_{\tau}||^2+M||u^2_{\tau}||^2,$ which contradicts Proposition \ref{directdersincreasing}. The rest of the corollary is proved by applying Proposition \ref{directdersincreasing} in an analogous manner.
\end{proof}
\begin{corollary} \label{somecor2}
Given any $M, N>0$, one and only one of the following cases can occur:
\begin{itemize}
    \item[(i)] $||x^1_{\tau}(\delta_M)||<||x^1_{\tau}(\delta_N)||$ \text{(}resp.,  $||x^2_{\tau}(\delta_M)||<||x^2_{\tau}(\delta_N)||$ ) and \\ $||x^2_{\tau}(\delta_M)||>||x^2_{\tau}(\delta_N)||$ (resp., $||x^1_{\tau}(\delta_M)||>||x^1_{\tau}(\delta_N)||$);
      \item[(ii)] $x_{\tau}(\delta_M)=x_{\tau}(\delta_N)$.
\end{itemize}
\end{corollary}
\begin{proof}  
Suppose first that  $||x^1_{\tau}(\delta_M)||<||x^1_{\tau}(\delta_N)||$. If 
$||x^2_{\tau}(\delta_M)||\le ||x^2_{\tau}(\delta_N)||$, one obtains 
$$K_{\tau}(N)>||x^1_{\tau}(\delta_M)||^2+N||x^2_{\tau}(\delta_M)||^2,$$
 which contradicts Proposition \ref{directdersincreasing}. Thus, $||x^2_{\tau}(\delta_M)||>||x^2_{\tau}(\delta_N)||$. The proof in the case that $||x^2_{\tau}(\delta_M)||<||x^2_{\tau}(\delta_N)||$ proceeds in an entirely analogous manner. \par 
 Now, assume $||x^1_{\tau}(\delta_M)||=||x^1_{\tau}(\delta_N)||.$ If $||x^2_{\tau}(\delta_M)||<||x^2_{\tau}(\delta_N)||,$ then we again obtain $K_{\tau}(N)>||x^1_{\tau}(\delta_M)||^2+N||x^2_{\tau}(\delta_M)||^2,$ a contradiction (the inequality $||x^2_{\tau}(\delta_M)||>||x^2_{\tau}(\delta_N)||$ can be ruled out in an analogous way). Thus, we must have $||x_{\tau}(\delta_M)||=||x_{\tau}(\delta_N)||$, which, by Proposition \ref{directdersincreasing}, gives us  $x_{\tau}(\delta_M)=x_{\tau}(\delta_N)$.
\end{proof}
Our next proposition (while fitting the theme of this section) will not be used in the sequel, so we record it without a proof (one can use Proposition \ref{directdersincreasing} in combination with the previous two lemmas).
\begin{proposition}
Assume $\tau\in\mathbb{T}^2.$ Then,
 $$\lim_{M\to 0+}||x^1_{\tau}(\delta_M)||=\lim_{M\to 0+}\sqrt{K_{\tau}(M)}=\inf_{M>0}\big\{||x^1_{\tau}(\delta_M)||\big\} $$
 (resp., $\lim_{M\to +\infty}||x^2_{\tau}(\delta_M)||=\lim_{M\to +\infty}\sqrt{K_{\tau}(M)}=\inf_{M>0}\big\{||x^2_{\tau}(\delta_M)||\big\} $). \\
 Moreover, if there exist $u\in Y_{\tau}$ and a sequence $\{M_k\}$ of positive numbers such that $M_k\to 0$ (resp. $M_k\to\infty$) and $||u^1_{\tau}||\le ||x^1_{\tau}(\delta_{M_k})||$ (resp., $||u^2_{\tau}||\le ||x^2_{\tau}(\delta_{M_k})||$) for all $k$, then $$\lim_k x^1_{\tau}(\delta_{M_k})=u^1_{\tau} \hspace{0.3 cm} \text{in norm}$$
 (resp.,  $\lim_k x^2_{\tau}(\delta_{M_k})=u^2_{\tau}$ in norm). 
\end{proposition}

We now prove a theorem that describes those vectors in $Y_{\tau}$ with null components (recall that, by Lemma \ref{twonullcoords}, these vectors, if they exist, must be unique). 
\begin{theorem}\label{BbutnotC}
Assume the B-point $\tau\in\mathbb{T}^2$ is such that there exists $u_{\tau}\in Y_{\tau}$ with $u^2_{\tau}=0$. Then, $$\lim_{M\to+\infty}K_{\tau}(M)=||u^1_{\tau}||^2$$ and also
    $$\lim_{M\to+ \infty}x_{\tau}(\delta_M)=u_{\tau}$$ in norm. 
 
Moreover, $\tau$ is a C-point for $\phi$ if and only if $X_{\tau}=\{u_{\tau}\}$. In this case, every $v_{\tau}\in Y_{\tau}$ such that $v_{\tau}\neq u_{\tau}$ must satisfy $||v^1_{\tau}||>||u^1_{\tau}||$ and $v^2_{\tau}\neq 0$.
\end{theorem}
\begin{proof}
We prove the C-point portion of the theorem first. If $X_{\tau}=\{u_{\tau}\},$ Theorem \ref{facialB} implies that $\tau$ is a C-point. Conversely,  assume that $\tau$ is a C-point. Write $X_{\tau}=\{x_{\tau}\}$. Proposition \ref{directdersincreasing} then implies that 
$$||x^1_{\tau}||^2+M||x^2_{\tau}||^2\le ||u^1_{\tau}||^2,$$
for all $M>0$. Thus, $x^2_{\tau}=0$. Lemma \ref{twonullcoords} then gives us $x_{\tau}=u_{\tau}$ and we can also write 
\begin{equation}\label{anothermodel}
 1-\phi(\lambda)\overline{\phi(\tau)}=(1-\lambda^1\overline{\tau^1})\langle u^1_{\lambda}, x^1_{\tau}\rangle \hspace {0.3 cm} (\lambda\in\DD^2).   
\end{equation}
Assume $x^1_{\tau}\neq 0$ (else the result will be be trivial) and  let $v_{\tau}\in Y_{\tau}.$  Lemma \ref{comparisonsetup} implies that 
$$\langle x^1_{\tau}, v^1_{\tau}\rangle=||x^1_{\tau}||^2.$$
Thus, either $||v^1_{\tau}||>||x^1_{\tau}||,$ or $||v^1_{\tau}||^2=||x^1_{\tau}||^2=\langle x^1_{\tau}, v^1_{\tau}\rangle,$ which leads to $v^1_{\tau}=x^1_{\tau}$. But then, comparing 
$$ 1-\phi(\lambda)\overline{\phi(\tau)}=(1-\lambda^1\overline{\tau^1})\langle u^1_{\lambda}, v^1_{\tau}\rangle+(1-\lambda^2\overline{\tau^2})\langle u^2_{\lambda}, v^2_{\tau}\rangle $$
with (\ref{anothermodel}) gives us $v^ 2_{\tau}=0,$ thus $v_{\tau}=x_{\tau}$ as desired. Finally, if $v^2_{\tau}=0,$ we can apply Lemma \ref{twonullcoords} to conclude that  $v_{\tau}=x_{\tau}$.

\par 
Now, we prove the first part of the theorem. If $\tau$ is a C-point, the theorem follows by our previous result. So, assume that $\tau$ is not a C-point, in which case Proposition \ref{directdersincreasing} implies that $K_{\tau}(M)$ is strictly increasing. The bound $K_{\tau}(M)\le ||u^1_{\tau}||^2$, for every $M>0,$ implies that $\lim_{M\to+\infty}x^2_{\tau}(\delta_M)=0$ in norm and also that $\big\{||x^1_{\tau}(\delta_M)||\big\}$ is bounded with respect to $M$. \par 
Now, let $\{M_k\}$ be any sequence converging to $+\infty$ such that $x^1_{\tau}(\delta_{M_k})$ converges to some $x^1\in M^1$ weakly. Also, fix a decreasing null sequence $\{\epsilon_k\}.$ In view of Theorem \ref{generaldir}, we can find $\{\lambda_k\}\subset\DD^2$ that converges to $\tau$ and such that $|\phi(\lambda_k)-\phi(\tau)|<\epsilon_k$ and also $||x_{\tau}(\delta_{M_k})-u_{\lambda_k}||<\epsilon_k$, for all $k.$ Thus, $\lim_k \phi(\lambda_k)=\phi(\tau)$, $u^1_{\lambda_k}$ converges weakly to $x^1$ and $u^2_{\lambda_k}$ converges to $0$ in norm. Now, the model formula (\ref{basicmodelform}) implies that 
$$1-\phi(\lambda)\overline{\phi(\lambda_k)}=(1-\lambda^1\overline{\lambda^1_k})\langle u^1_{\lambda}, u^1_{\lambda_k}\rangle+(1-\lambda^2\overline{\lambda^2_k})\langle u^2_{\lambda}, u^2_{\lambda_k}\rangle,$$
for all $k$ and $\lambda\in\DD^2.$ Letting $k\to\infty$ then gives us, 
\begin{equation} \label{compare}
  1-\phi(\lambda)\overline{\phi(\tau)}=(1-\lambda^1\overline{\tau^1})\langle u^1_{\lambda}, x^1\rangle,  
\end{equation}
for all $\lambda. $ However, since $u_{\tau}\in Y_{\tau}$ we can also write (in view of (\ref{basicmodelsetupeq}))
$$  1-\phi(\lambda)\overline{\phi(\tau)}=(1-\lambda^1\overline{\tau^1})\langle u^1_{\lambda}, u^1_{\tau}\rangle,  $$
for all $\lambda.$ Comparing this equality with (\ref{compare}) then gives us 
$\langle u^1_{\lambda}, u^1_{\tau}\rangle=\langle u^1_{\lambda}, x^1\rangle$ for all $\lambda.$ Since both vectors $u^1_{\tau}, x^1$ are weak limits of elements from $\{u^1_{\lambda} : \lambda\in\DD^2\}$, we may conclude that $u^1_{\tau}=x^1$.  But then, observe that (by a standard property of weak limits)
$$||u^1_{\tau}||^2=||x^1||^2$$ $$\le \liminf_k ||x^1_{\tau}(\delta_{M_k})||^2$$ 
$$\le \limsup_k ||x^1_{\tau}(\delta_{M_k})||^2$$
$$\le \limsup_k K_{\tau}(M_k)  $$
$$\le||u^1_{\tau}||^2,$$
which implies that $\lim_k x^1_{\tau}(\delta_{M_k})=x^1$ in norm and also that $\lim_k K_{\tau}(M_k)=||u^1_{\tau}||^2$. We conclude that $x^1_{\tau}(\delta_M)$  converges to $u^1_{\tau}$ in norm and also that $\lim_{M\to+\infty}K_{\tau}(M)=||u^1_{\tau}||^2,$ as desired.
 \end{proof}

We end this section with a weighted version of Julia's inequality for the bidisk, which will be of critical importance in Section \ref{refineHerve}. Our methods are motivated by the proof of \cite[Theorem 4.9]{AMYcaratheodory}.

\begin{theorem}\label{Juliarevisited}
    Let $\phi\in\mathcal{S}_2$ and  $ \tau=(\tau^1, \tau^2)\in\mathbb{T}^2.$ Assume also that $\alpha$ and $M$ are positive numbers. The following assertions are equivalent:
    \begin{itemize}
        \item[(i)] $\tau$ is a B-point for $\phi$ and $$\frac{D_{-(\tau^1,\tau^2M)}\phi(\tau)}{-\phi(\tau)}\le \alpha;$$
        \item[(ii)] There exists a sequence $\{\lambda_n\}\subset\DD^2$ such that $\lambda_n\to\tau,$ $\lim_n\frac{1-|\lambda^2_n|}{1-|\lambda^1_n|}=M$ and 
        $$ \lim_n\frac{1-|\phi(\lambda_n)|}{1-||\lambda_n||}\le \begin{cases}
  \alpha \hspace{0.65 cm} \text{ if }M\ge 1; \\ 
  \frac{\alpha}{M} \hspace{0.5 cm} \text{ if }M<1; 
        \end{cases} $$
        \item[(iii)] There exists $\omega\in\mathbb{T}$ such that $$\phi(E(\tau, R_1, R_2))\subset E(\omega, \max\{aR_1, aR_2/M\}), \hspace{0.3 cm} \forall R_1, R_2>0.$$
    \end{itemize}
    If (iii) holds, $\omega$ will necessarily be equal to $\phi(\tau).$
\end{theorem}
\begin{proof}
Let $(M, u)$ be a model for $\phi.$ \par 
First, we show that (i) implies (ii). Assuming (i) holds, set $\delta_M=(\tau^1, \tau^2 M)$ and fix a decreasing null sequence $\{r_n\}.$ Since $\tau$ is a B-point, Theorem \ref{generaldir} allows us to deduce that
\begin{equation} \label{sth}
||x^1_{\tau}(\delta_M)||^2+M||x^2_{\tau}(\delta_M)||^2\le \alpha,
\end{equation}
 and also $\lim_n u_{\tau-r_n\delta_M}=x_{\tau}(\delta_M)$ (in norm). Now, assume $M\ge 1$ and put $\lambda_n=\tau-r_n\delta_M.$ (\ref{basicmodelform}) allows us to write:
 $$\lim_n\frac{1-|\phi(\lambda_n)|}{1-||\lambda_n||}=\lim_n\frac{1-|\phi(\lambda_n)|^2}{1-||\lambda_n||^2} $$
$$=\lim_n\frac{1-|\phi(\lambda_n)|^2}{1-|\lambda^1_n|^2}  $$
 $$=\lim_n \bigg(||u^1_{\lambda_n}||^2+ \frac{1-|\lambda^2_n|^2}{1-|\lambda^1_n|^2}||u^2_{\lambda_n}||^2   \bigg)$$
 $$=||x^1_{\tau}(\delta_M)||^2+M||x^2_{\tau}(\delta_M)||^2 $$
 $$\le \alpha.$$
 The proof for $M<1$ is entirely analogous and is omitted. \par 
 Next, we show that (ii) implies (iii). The assumptions in (ii) clearly imply that $\tau$ is a B-point for $\phi.$ But then, we can argue as above to deduce that (\ref{sth}) holds. Thus, we can use (\ref{basicmodelsetupeq}) to obtain
 $$|1-\phi(\lambda)\overline{\phi(\tau)}|\le |1-\lambda^1\overline{\tau^1}|\cdot||x^1{(\delta_M)}||\cdot||u^1_{\lambda}||+ |1-\lambda^2\overline{\tau^2}|\cdot||x^2{(\delta_M)}||\cdot||u^2_{\lambda}||,$$
 for all $\lambda\in\DD^2.$
Setting $R_j=\frac{|\tau^j-\lambda^j|^2}{1-|\lambda^j|^2}$ ($j=1, 2$) and applying Cauchy-Schwarz then gives us 
$$|\phi(\tau)-\phi(\lambda)|^2\le \big(||x^1{(\delta_M)}||^2+M||x^2{(\delta_M)}||^2 \big)\bigg(|\tau^1-\lambda^1|^2||u^1_{\lambda}||^2+\frac{|\tau^2-\lambda^2|^2}{M}||u^2_{\lambda}||^2  \bigg) $$
$$\le \alpha \max\{R_1, R_2/M \}\big((1-|\lambda^1|^2)||u^1_{\lambda}||^2+(1-|\lambda^2|^2)||u^2_{\lambda}||^2 \big) $$
$$=\max\{\alpha R_1, \alpha R_2/M\}\big(1-|\phi(\lambda)|^2\big),$$
which implies 
$$\frac{|\phi(\tau)-\phi(\lambda)|^2}{1-|\phi(\lambda)|^2}\le \max\{\alpha R_1, \alpha R_2/M\} \hspace{0.4 cm} (\lambda\in\DD^2)$$
and our proof is complete. \par 
Lastly, we show that (iii) implies (i). Set $\lambda_n=\tau-r_n\delta_M$, where $\{r_n\}$ is a decreasing null sequence. Assuming (iii) holds, we obtain  
$$\frac{|\omega-\phi(\lambda_n)|^2}{1-|\phi(\lambda_n)|^2}\le \alpha \max\bigg\{\frac{|\tau^1-\lambda^1_n|^2}{1-|\lambda^1_n|^2},\frac{1}{M}\frac{|\tau^2-\lambda^2_n|^2}{1-|\lambda^2_n|^2}   \bigg\} $$
$$=\alpha\max\bigg\{\frac{r_n}{2-r_n},  \frac{r_n}{2-Mr_n}\bigg\}$$
$$=R_n.$$
Thus,  we can write 
$$\phi(\lambda_n)\in \text{cl}\big(E(\omega, R_n))=\text{cl}\bigg(D\bigg(\frac{\omega}{R_n+1} ,\frac{R_n}{R_n+1}\bigg)\bigg),$$
for all $n\ge 1.$ \par 
Now, assume $M\ge 1$ (the proof in the case where $M<1$ will be entirely analogous). Then, $R_n=\frac{\alpha r_n}{2-Mr_n}$, $||\lambda_n||=1-r_n$ and we can compute 
$$\frac{1-|\phi(\lambda_n)|}{1-||\lambda_n||}=\frac{1-|\phi(\lambda_n)|}{r_n} $$
$$\le \frac{|\phi(\lambda_n)-\omega|}{r_n} $$
$$\le \frac{2}{r_n}\frac{R_n}{R_n+1} $$
$$=\frac{2\alpha}{2+(\alpha-M)r_n}\to \alpha, $$
as $n\to\infty.$ This implies that $\tau$ is a B-point for $\phi$ and $\omega=\phi(\tau)$. Also, since 
$$\lim_n\frac{1-|\phi(\lambda_n)|}{1-||\lambda_n||}=||x^1_{\tau}(\delta_M)||^2+M||x^2_{\tau}(\delta_M)||^2,$$
we obtain that (\ref{sth}) holds. Theorem  \ref{generaldir} then finishes off the proof.
\end{proof}
\begin{remark}
In this theorem, we only considered points $\tau$ in the distinguished boundary. For facial B-points, the situation is more straightforward; see \cite[Theorem 3.2]{AMYfacial}.
\end{remark}
\begin{remark}
Observe that if we assume $$\lim_{r\to 1-}\frac{1-|\phi(r\tau)|}{1-||r\tau||}=\liminf_{\lambda\to\tau}\frac{1-|\phi(\lambda)|}{1-||\lambda||}\le \alpha,$$
we obtain that (ii) holds with $M=1.$ Hence, 
$$\phi(E(\tau, R, R))\subset E(\phi(\tau), \alpha\max\{R, R\})=E(\phi(\tau), \alpha R),$$
for all $R>0,$ which is the usual statement of Julia's inequality over the bidisk. 
\end{remark}
\begin{remark}\label{aremark}
Julia-type inequalities like the one in Theorem \ref{Juliarevisited}(iii) were also considered by Frosini in \cite{FrosiniBusemann}, where she used Busemann sublevel sets to obtain analogous statements. Specifically, her Julia-type lemma \cite[Theorem 1]{FrosiniBusemann} depends on the behavior of $\phi$ along chosen complex
geodesics that approach the boundary point $\tau$. Theorem \ref{Juliarevisited} can then be viewed as a refinement of that result, as it essentially says that every ``weighted" version of Julia's inequality is equivalent to an inequality involving certain directional derivatives of $\phi$ at the corresponding boundary B-point. 
\end{remark}

 \small
\section{CRITERIA FOR DENJOY-WOLFF POINTS} \label{charactDW}
\large  We will now use our work from Section \ref{generalBCdir} to study Type I/II DW points, as defined in Section \ref{prelims}. \par 
We start with two lemmas.
\begin{lemma}\label{70}
Let $\phi\in\mathcal{S}_2$ with model $(M, u)$ and assume $\xi(\mu), \mu$ are points in $\DD$ such that $\phi(\xi(\mu), \mu)=\xi(\mu).$ Then,  $||u^1_{(\xi(\mu),\mu)}||\le 1$. Also, $||u^1_{(\xi(\mu),\mu)}||= 1$ if and only if $u^2_{(\xi(\mu),\mu)}=0.$
\end{lemma}
\begin{proof}
In (\ref{basicmodelform}), set $(\lambda^1, \lambda^2)=(\mu^1, \mu^2)=(\xi(\mu), \mu)$ to obtain
$$1-|\xi(\mu)|^2=1-|\phi(\xi(\mu), \mu)|^2=(1-|\xi(\mu)|^2)||u^1_{(\xi(\mu), \mu)}||^2+(1-|\mu|^2)||u^2_{(\xi(\mu), \mu)}||^2.$$
Since $|\xi(\mu)|, |\mu|<1,$ the conclusions of the lemma follow easily.
\end{proof}
The next lemma is well-known (e.g. it appears as Theorem 2 in \cite{HERVE}). We include a proof for the sake of completeness.
\begin{lemma} \label{identityslices}
Assume $\phi\in\mathcal{S}_2$ and that there exists $\mu_0\in\DD$ such that the slice function $\phi_{\mu_0}$ is the identity on $\DD.$ Then, $\phi\equiv\pi_1$. 
\end{lemma}
\begin{proof}
Let $(M, u)$ be a model for $\phi.$ We can use (\ref{basicmodelform}) to obtain 
\begin{equation}\label{yesanothereq}
1=\frac{1-|\phi(\lambda, \mu_0)|^2}{ 1-|\lambda|^2}=||u^1_{(\lambda, \mu_0)}||^2+\frac{1-|\mu_0|^2}{ 1-|\lambda^2|}||u^2_{(\lambda, \mu_0)}||^2,   
\end{equation}
for all $\lambda\in\DD.$ Thus, $||u^1_{(\lambda, \mu_0)}||\le 1,$ for all $\lambda\in\DD,$ with equality if and only if $u^2_{(\lambda, \mu_0)}=0.$
\par Now, fix $\tau^1\in\mathbb{T}$ and let $\lambda\to\tau^1$ in (\ref{yesanothereq}) to obtain that $(\tau^1, \mu_0)$ is a B-point for $\phi,$ $\phi(\tau^1, \mu_0)=\tau^1$ and also that there exists $u_{(\tau^1, \mu_0)}\in Y_{(\tau^1, \mu_0)}$ such that $||u^1_{(\tau^1, \mu_0)}||\le 1$ and $u^2_{(\tau^1, \mu_0)}=0.$ (\ref{basicmodelsetupeq}) then implies that 
\begin{equation}\label{whatever}
1-\phi(\lambda, \mu)\overline{\tau^1}=(1-\lambda\overline{\tau^1})\langle u^1_{(\lambda, \mu)}, u^1_{(\tau^1, \mu_0)} \rangle,   
\end{equation}
for all $\lambda, \mu\in\DD.$ Setting $\mu=\mu_0$ then gives us 
$$1-\lambda\overline{\tau^1}=(1-\lambda\overline{\tau^1})\langle u^1_{(\lambda, \mu_0)}, u^1_{(\tau^1, \mu_0)} \rangle,$$
hence $$\langle u^1_{(\lambda, \mu_0)}, u^1_{(\tau^1, \mu_0)} \rangle=1\ge ||u^1_{(\lambda, \mu_0)}||^2, ||u^1_{(\tau^1, \mu_0)}||^2.$$
 This implies that $u^1_{(\lambda, \mu_0)}=u^1_{(\tau^1, \mu_0)}$
 (and both have to be unit vectors) and also $u^2_{(\lambda, \mu_0)}=0$, for all $\lambda\in\DD$ . Hence,
 $$1-\phi(\lambda, \mu)\overline{\lambda'}=1-\phi(\lambda, \mu)\overline{\phi(\lambda', \mu_0)}$$ $$=(1-\lambda\overline{\lambda'})\langle u^1_{(\lambda, \mu)}, u^1_{(\lambda', \mu_0)}\rangle=$$ $$=(1-\lambda\overline{\lambda'})\langle u^1_{(\lambda, \mu)}, u^1_{(\tau^1, \mu_0)}\rangle,$$
 for all $\lambda, \mu, \lambda'\in\DD.$ Since both sides are affine functions of $\lambda'$, we obtain $\phi(\lambda, \mu)=\lambda,$ for all $\lambda, \mu\in\DD,$ as desired.
\end{proof}
Now, we use model theory to give a new proof of the fact that every $\phi\in\mathcal{S}_2$ (that is not a coordinate projection) must either be a Type I or a Type II function, a result originally due to Herv\'{e} (see \cite[Theorem 1]{HERVE}).
\begin{theorem} \label{130}
Every $\phi\in\mathcal{S}_2$ such that $\phi\neq \pi^1$ (resp., $\phi\neq \pi^2$) is either a left Type I (resp. right Type I) or a left Type II (resp. right Type II) function.
\end{theorem}
\begin{proof} First, we prove the left Type I/II version of the theorem. Note that, since $\phi\neq\pi^1$, Lemma \ref{identityslices} implies that $\phi_{\mu}$ is not the identity on $\DD$, for any $\mu\in\DD.$ Thus, every such slice function will have a unique Denjoy-Wolff point (either on the interior of the disk or on the boundary). 
\par To begin, assume that there exists some $\mu_0\in\DD$ such that the slice $\phi_{\mu_0}$ has its Denjoy-Wolff point $\tau^1$ on $\mathbb{T}$. Let $\lambda_n=\rho_n \tau^1,$ where $\{\rho_n\}$ is an increasing sequence of positive numbers tending to $1.$ By the single-variable theory of Denjoy-Wolff points, we have $\lim_n\phi_{\mu_0}(\lambda_n)=\tau^1$ and
$$\frac{1-|\phi(\lambda_n, \mu_0)|^2}{1-||(\lambda_n, \mu_0)||^2}=\frac{1-|\phi_{\mu_0}(\lambda_n)|^2}{1-|\lambda_n|^2}\to \alpha_{\mu_0}\le1,$$
as $n\to\infty.$ Thus, $(\tau^1, \mu_0)$ is a B-point for $\phi.$ Using the model formula for $\phi,$ we also see that 
$$||u^1_{(\lambda_n, \mu_0)}||^2+\frac{1-|\mu_0|^2}{1-|\lambda_n|^2}||u^2_{(\lambda_n, \mu_0)}||^2=\frac{1-|\phi_{\mu_0}(\lambda_n)|^2}{1-|\lambda_n|^2}, \hspace{0.3 cm} \forall n\ge 1.$$
Letting $n\to\infty$ and taking into account that $\frac{1-|\mu_0|^2}{1-|\lambda_n|^2}\to\infty$, we obtain the existence of $u_{(\tau^1, \mu_0)}\in Y_{(\tau^1, \mu_0)}$ satisfying $||u^1_{(\tau^1,\mu_0)}||\le \alpha_{\mu_0}\le 1$ and $u^2_{(\tau^1,\mu_0)}=0$. In view of (\ref{basicmodelsetupeq}), we can write
\begin{equation}\label{81}
1-\phi(\lambda, \mu)\overline{\tau^1}=(1-\lambda\overline{\tau^1})\langle u^1_{(\lambda, \mu)}, u^1_{(\tau^1, \mu_0)}\rangle,
\end{equation}
for all $\lambda, \mu\in \DD.$ \par Now, assume there exists some slice $\phi_{\mu_1}$ such that $\mu_1\neq \mu_0$ and $\phi_{\mu_1}$ has an interior fixed point $p\in\DD.$ Set $(\lambda, \mu)=(p, \mu_1)$ in (\ref{81}) to obtain 
$$1-p\overline{\tau^1}=(1-p\overline{\tau^1})\langle u^1_{(p, \mu_1)}, u^1_{(\tau^1, \mu_0)}\rangle,$$
hence $\langle u^1_{(p, \mu_1)}, u^1_{(\tau^1, \mu_0)}\rangle=1$. Lemma (\ref{70}) then implies that $u^1_{(p, \mu_1)}=u^1_{(\tau^1, \mu_0)}$ (and both will be unit vectors) and $u^2_{(p, \mu_1)}=0$. Thus, we may substitute $(\mu^1, \mu^2)=(p, \mu_1)$ in (\ref{basicmodelform}) to obtain
\begin{equation}\label{82}
1-\phi(\lambda, \mu)\overline{p}=(1-\lambda\overline{p})\langle u^1_{(\lambda, \mu)}, u^1_{(p, \mu_1)}\rangle=(1-\lambda\overline{p})\langle u^1_{(\lambda, \mu)}, u^1_{(\tau^1, \mu_0)}\rangle,
\end{equation}
for all $\lambda, \mu \in\DD. $ Comparing (\ref{82}) with (\ref{81}) then allows us to deduce that $\phi$ is equal to the identity, a contradiction. \par 
So far, we have proved that every slice function $\phi_{\mu}$ must have its Denjoy-Wolff point on the boundary of $\DD$ (under the assumption that at least one of them does). We now show that $\tau^1$ (the Denjoy-Wolff point of the slice $\phi_{\mu_0}$ we started with) is actually the Denjoy-Wolff point of all slices $\phi_{\mu}.$  Indeed, suppose we can find a slice $\phi_{\mu_1}$ with a different Denjoy-Wolff point $\sigma^1\in\mathbb{T}.$ Arguing as in the beginning of the proof, we obtain that $(\sigma^1, \mu_1)$ is a B-point for $\phi$, its value at $(\sigma^1, \mu_1)$ is $\sigma^1$ and also there exists $u_{(\sigma^1, \mu_1)}\in Y_{(\sigma^1, \mu_1)}$ satisfying $||u^1_{(\sigma^1, \mu_1)}||\le 1$ and $u^2_{(\sigma^1, \mu_1)}=0$. (\ref{basicmodelsetupeq}) implies that
\begin{equation}\label{85}
1-\phi(\lambda, \mu)\overline{\sigma^1}=(1-\lambda\overline{\sigma^1})\langle u^1_{(\lambda, \mu)}, u^1_{(\sigma^1, \mu_1)}\rangle, \hspace{0.3 cm} \forall \lambda,\mu\in\DD.
\end{equation}
If in (\ref{85}) we let $(\lambda, \mu)\to (\tau^1, \mu_0)$ in such a way that $u^1_{(\lambda, \mu)}$ converges weakly to $u^1_{(\tau^1, \mu_0)}$, we obtain (since $\phi(\tau^1, \mu_0)=\tau^1$ and $\sigma^1\neq \tau^1$) $$\langle  u^1_{(\sigma^1, \mu_1)},  u^1_{(\tau^1,\mu_0)}\rangle=1\ge || u^1_{(\sigma^1, \mu_1)}||^2, ||u^1_{(\tau^1,\mu_0)}||^2,$$ hence $ u^1_{(\sigma^1, \mu_1)}= u^1_{(\tau^1,\mu_0)}$. Comparing (\ref{81}) with (\ref{85}) then gives us that $\phi$ is equal to the identity, a contradiction.
\par 
On the other hand, assume that every slice $\phi_{\mu}$ has a (necessarily unique) interior fixed point $\xi(\mu).$ To show that $\phi$ is a left Type II function, it suffices to prove that $\xi:\DD\to\DD$ is actually a holomorphic function. First, note that putting $(\lambda^1, \lambda^2)=(\xi(\mu), \mu)$ and $(\mu^1, \mu^2)=(\xi(\mu'), \mu')$ in (\ref{basicmodelform}) gives us 
$$
 1-\xi(\mu)\overline{\xi(\mu')}$$ \begin{equation} \label{45}=(1-\xi(\mu)\overline{\xi(\mu')}) \langle u^1_{(\xi(\mu), \mu)}, u^1_{(\xi(\mu'), \mu')}\rangle  +(1-\mu\overline{\mu'}) \langle u^2_{(\xi(\mu), \mu)}, u^2_{(\xi(\mu'), \mu')}\rangle,
\end{equation}
for all $\mu, \mu'\in\DD.$ \par 
  Now, if $||u^1_{(\xi(\mu'), \mu')}||=1$ for some $\mu'\in \DD$, the model formula for $\phi$ yields (since $u^2_{(\xi(\mu'), \mu')}=0$ in view of Lemma \ref{70})
\begin{equation}\label{51} 1-\phi(\lambda, \mu)\overline{\xi(\mu')}=(1-\lambda\overline{\xi(\mu')})\langle u^1_{(\lambda, \mu)}, u^1_{(\xi(\mu'), \mu')}\rangle ,
\end{equation} 
for all $\lambda, \mu.$ Plugging in $(\lambda, \mu)=(\xi(\mu), \mu)$ gives us $\langle u^1_{(\xi(\mu), \mu)}, u^1_{(\xi(\mu'), \mu')}\rangle=1$, for all $\mu,$ hence $u^1_{(\xi(\mu), \mu)}=u^1_{\xi}=$constant (of norm $1$) and $u^2_{(\xi(\mu), \mu)}=0$ for all $\mu$. Thus, we obtain
\begin{equation} \label{52} 1-\phi(\lambda, \mu)\overline{\xi(\sigma)}=(1-\lambda\overline{\xi(\sigma)})\langle u^1_{(\lambda, \mu)}, u^1_{\xi}\rangle , \end{equation}
for all $\lambda, \mu,\sigma\in\DD.$  \par 
There are now two separate cases to examine. Either $\phi(\lambda, \mu)=\lambda\langle u^1_{(\lambda, \mu)}, u^1_{\xi}\rangle$ for all $\lambda, \mu,$ in which case (\ref{52}) implies that $\langle u^1_{(\lambda, \mu)}, u^1_{\xi}\rangle=1$ (for all $\lambda, \mu$), hence $\phi=\pi^1,$ a contradiction, or we can find $\lambda_0, \mu_0\in\DD$ such that $\phi(\lambda_0, \mu_0)\neq\lambda_0\langle u^1_{(\lambda_0, \mu_0)}, u^1_{\xi}\rangle$. Then, (\ref{52}) implies that 
$$\overline{\xi(\sigma)}=\frac{\langle u^1_{(\lambda_0, \mu_0)}, u^1_{\xi}\rangle-1}{\lambda_0\langle u^1_{(\lambda_0, \mu_0)}, u^1_{\xi}\rangle-\phi(\lambda_0, \mu_0)},$$
for all $\sigma\in\DD.$ Thus, $\xi$ is constant (and trivialy holomorphic).
\par 
There is one more possibility to consider: suppose that $||u^1_{(\xi(\mu'), \mu')}||< 1$ for all $\mu'.$ (\ref{45}) then becomes
\begin{equation}\label{46}
 1-\xi(\mu)\overline{\xi(\mu')}=(1-\mu\overline{\mu'})\frac{\langle u^2_{(\xi(\mu), \mu)}, u^2_{(\xi(\mu'), \mu')}\rangle}{1- \langle u^1_{(\xi(\mu), \mu)}, u^1_{(\xi(\mu'), \mu')}\rangle },
\end{equation}
for all $\mu, \mu'\in\DD$. In other words, $$\frac{1-\xi(\mu)\overline{\xi(\mu')}}{1-\mu\overline{\mu'}}$$ is the Schur product of the positive-semidefinite kernels $\langle u^2_{(\xi(\mu), \mu)}, u^2_{(\xi(\mu'), \mu')}\rangle$ and $\frac{1}{1- \langle u^1_{(\xi(\mu), \mu)}, u^1_{(\xi(\mu'), \mu')}\rangle}$ (the latter is actually a complete Pick kernel, see \cite[Chapter 8]{OldPick}), hence it must be positive semi-definite as well. Automatic holomorphy of models (see \cite[Proposition 2.32]{NewPick}) then implies that $\xi$ is a holomorphic function on $\DD,$ concluding the proof.
\par 
 The right Type I/II version of the theorem follows by applying the left Type I/II version to the function $\widetilde{\phi}:\DD^2\to\DD$ defined by $\widetilde{\phi}(\lambda)=\phi(\lambda^2, \lambda^1)$, for all $\lambda\in\DD^2$.
\end{proof}
Next, we provide criteria for Type I DW points, as stated in Section \ref{prelims}. Recall that, given $\phi\in\mathcal{S}_2$ with model $(M, u)$ and a B-point $\tau\in\mathbb{T}^2,$ we have defined $\delta_M=(\tau^1, \tau^2 M)$ and $$K_{\tau}(M)=||x^1_{\tau}(\delta_M)||^2+M||x^2_{\tau}(\delta_M)||^2,$$  for all $M>0.$
\begin{proof}[Proof of Theorem \ref{TypeIcharabridged}]
First, we show that (iii) implies (ii). Indeed, assume that $\tau^1$ is the common Denjoy-Wolff point of all slice functions $\phi_{\mu}$ and let $|\sigma|\le 1.$ We will show that $(\tau^1, \sigma)$ is a left Type I DW point for $\phi$.  \par Fix a sequence $\{\mu_n\}\subset\DD$ tending to $\sigma$. Now, since $\tau^1$ is the Denjoy-Wolff point of $\phi_{\mu},$ we obtain that $\tau^1$ is a B-point for $\phi_{\mu},$ $\phi_{\mu}(\tau^1)=\tau^1$ and also 
$$\lim_{\lambda\xrightarrow[]{\text{nt}}\tau^1}\frac{1-|\phi_{\mu}(\lambda)|^2}{1-|\lambda|^2}\le 1,$$
for all $\mu\in\DD.$ Thus, it is possible to choose a sequence $\{\lambda_n\}\subset\DD$ converging to $\tau^1$ nontangentially, and sufficiently fast, so that we obtain $\lim_{n}\phi_{\mu_n}(\lambda_n)=\tau^1$, $\lim_n\frac{1-|\lambda_n|^2}{1-|\mu_n|^2}=0$ and also 
\begin{equation}\label{sthtypeI}
 \limsup_n\frac{1-|\phi(\lambda_n, \mu_n)|^2}{1-||(\lambda_n, \mu_n)||^2}=\limsup_n \frac{1-|\phi_{\mu_n}(\lambda_n)|^2}{1-|\lambda_n|^2}\le 1,   
\end{equation}
which implies that $(\tau^1, \sigma)$ is a B-point for $\phi$ and also $\phi(\tau^1, \sigma)=\tau^1.$ Moreover, the model formula for $\phi$ tells us 
$$||u^1_{(\lambda_n, \mu_n)}||^2+\frac{1-|\mu_n|^2}{1-|\lambda_n|^2}||u^2_{(\lambda_n, \mu_n)}||^2=\frac{1-|\phi(\lambda_n, \mu_n)|^2}{1-|\lambda_n|^2},$$
for all $n.$ Letting $n\to\infty$ and taking into account the limits $\lim_n\frac{1-|\lambda_n|^2}{1-|\mu_n|^2}=0$ and (\ref{sthtypeI}), we can deduce the existence of $u_{(\tau^1, \sigma)}\in Y_{(\tau^1, \sigma)}$ such that $||u^1_{(\tau^1, \sigma)}||\le 1$ and $u^2_{(\tau^1, \sigma)}=0$. This implies that $(\tau^1, \sigma)$ is a left Type I DW point for $\phi$. Since $\sigma$ was arbitrary, (ii) has been established.  \par 
That (ii) implies (i) is obvious. \par 
Now, we prove that (i) implies (iii). So, assume that there exists $|\sigma|\le 1$ such that $(\tau^1, \sigma)$ is a B-point for $\phi,$ $\phi(\tau^1, \sigma)=\tau^1$ and also there exists $u_{(\tau^1, \sigma)}\in Y_{(\tau^1, \sigma)}$ such that $||u^1_{(\tau^1, \sigma)}||\le 1$ and $u^2_{(\tau^1, \sigma)}=0$. We obtain 
\begin{equation}\label{slicemodel}
1-\phi_{\mu}(\lambda)\overline{\tau^1}=1-\phi(\lambda, \mu)\overline{\tau^1}=(1-\lambda\overline{\tau^1})\langle u^1_{(\lambda, \mu)}, u^1_{(\tau^1, \sigma)}\rangle,\end{equation}
for all $\lambda, \mu\in\DD.$ If we fix $\mu$, we may repeat the proof of ``(ii) implies (iii)" from Theorem \ref{Juliarevisited} to obtain  
$$\frac{|\tau^1-\phi_{\mu}(\lambda)|^2}{1-|\phi_{\mu}(\lambda)|^2}=\frac{|\tau^1-\phi(\lambda, \mu)|^2}{1-|\phi(\lambda, \mu)|^2}\le\alpha_{\sigma}\frac{|\tau^1-\lambda|^2}{1-|\lambda|^2},$$
for all $\lambda, \mu\in\DD,$ where $\alpha_{\sigma}=||u^1_{(\tau^1, \sigma)}||^2$.
Such an equality is then known to imply (see Section \ref{prelims}) that $\tau^1$ is a B-point for $\phi_{\mu}$, $\phi_{\mu}(\tau^1)=\tau^1$ and also that the angular derivative of $\phi_{\mu}$ at $\tau^1$ is equal to $\alpha_{\sigma}\le 1$, for all $\mu\in\DD.$ Since we also know (in view of Lemma \ref{identityslices}) that $\phi_{\mu}\neq \text{Id}_{\DD}$, for all $\mu\in\DD,$ we can conclude that $\tau^1$ is the common Denjoy-Wolff point of every slice function, i.e. (iii) holds. \par
Before we proceed, a few important observations are in order. Our previous arguments show that, if at least one point in the closed face $\{\tau^1\}\times\text{cl}(\DD)$ is a left Type I DW point for $\phi,$ then for every $|\sigma|\le 1$ there exists $u_{(\tau^1, \sigma)}=(u^1_{(\tau^1, \sigma)}, 0)\in Y_{(\tau^1, \sigma)}$ such that $||u^1_{(\tau^1, \sigma)}||\le 1$ and also (\ref{slicemodel}) holds, for all $\lambda, \mu\in\DD.$ Since $\sigma$ was arbitrary, (\ref{slicemodel})  implies that the vectors $u^1_{(\tau^1, \sigma)}$ do not actually depend on $\sigma,$ thus
$u_{(\tau^1, \sigma)}=u_{\tau^1}=(u^1_{\tau^1}, 0)$ for all $\sigma\in\text{cl}(\DD).$
In particular, letting $\phi'_{\mu}(\tau^1)$ denote the angular derivative of $\phi_{\mu}$ at $\tau^1,$ we obtain 
\begin{equation}\label{ang}
\phi'_{\mu}(\tau^1)=||u_{\tau^1}||^2\le 1,    
\end{equation}
for all $|\mu |< 1.$ Also, notice that, in view of Lemma \ref{twonullcoords}, $u_{\tau^1}$ will be the unique vector in $Y_{(\tau^1, \sigma)}$ with $M^2$-component equal to $0$, for all $|\sigma|\le 1.$ \par 
Next, we show that (iii) implies (v). Fix an arbitrary $\sigma\in\text{cl}(\DD)$. By our previous results, $(\tau^1, \sigma)$ is a B-point for $\phi,$ $\phi(\tau^1, \sigma)=\tau^1$ and also there exists $u_{\tau^1}=(u^1_{\tau^1}, 0)\in Y_{(\tau^1, \sigma)}$ (not depending on $\sigma$) such that $||u_{\tau^1}||\le 1.$ If we also assume $|\sigma|<1,$ then $(\tau^1, \sigma)$ is a facial B-point, so \cite[Theorem 3.2]{AMYfacial} implies that $Y_{(\tau^1, \sigma)}=\{u_{\tau^1}\}$ and also
$$\frac{D_{-(\tau^1, \sigma M)}\phi(\tau^1, \sigma)}{-\tau^1}=\frac{D_{-(\tau^1, \sigma M)}\phi(\tau^1, \sigma)}{-\phi(\tau^1, \sigma)}=||u_{\tau^1}||^2\le 1,$$
for all $M>0,$ as desired. On the other hand, assume that $|\sigma|=1$. We may apply Theorems \ref{generaldir} and \ref{BbutnotC} to obtain that 
$$\frac{D_{-(\tau^1, \sigma M)}\phi(\tau^1, \sigma)}{-\tau^1}=K_{(\tau^1, \sigma)}(M)\le ||u_{\tau^1}||^2\le 1,$$ 
for all $M>0$. Actually, one can deduce the even stronger statement
$$\lim_{M\to\infty}\frac{D_{-(\tau^1, \sigma M)}\phi(\tau^1, \sigma)}{-\tau^1}=\lim_{M\to\infty} K_{(\tau^1, \sigma)}(M)= ||u_{\tau^1}||^2=\phi'_{\mu}(\tau^1),$$
for all $\mu\in\DD.$ Since $\sigma$ was arbitrary, we have established (v). \par 
That (v) implies (iv) is evident, so all that remains is to show that (iv) implies (iii). So, assume there exists $(\tau^1, \sigma)\in\mathbb{T}\times\text{cl}(\DD)$  such that the assumptions of (iv) are satisfied. If $|\sigma|<1, $ then $Y_{(\tau^1, \sigma)}=\{(u^1_{(\tau^1, \sigma)}, 0)\}$ and for any $M>0$ we have
$$||u_{(\tau^1, \sigma)}||^2=\frac{D_{-(\tau^1, \sigma M)}\phi(\tau^1, \sigma)}{-\tau^1}\le 1.$$
This shows that $(\tau^1, \sigma)$ is a left Type I DW point for $\phi,$ which gives us (i), hence (iii) holds. On the other hand, assume $|\sigma|=1.$ Fix an increasing sequence $\{M_k\}$ tending to $\infty.$ Since, by assumption, we have 
$$\frac{D_{-(\tau^1, \sigma M_k)}\phi(\tau^1, \sigma)}{-\tau^1}\le 1,$$
for all $k,$ Theorem \ref{Juliarevisited} implies that 
$$\phi\big(E((\tau^1, \sigma), R_1, R_2)\big)\subset E(\tau^1, \max\{R_1, R_2/M_k \}), $$
for all $k\ge 1$ and $R_1, R_2>0.$ Letting $k\to\infty$ yields
$$\phi\big(E((\tau^1, \sigma), R_1, R_2)\big)\subset E(\tau^1, R_1\}), $$
for all $R_1, R_2>0,$ which translates into the inequality
$$\frac{|\tau^1-\phi_{\mu}(\lambda)|^2}{1-|\phi_{\mu}(\lambda)|^2}\le \frac{|\tau^1-\lambda|^2}{1-|\lambda|^2},$$
for all $\lambda, \mu\in\DD.$ As already mentioned during the proof of ``(i) implies (iii)", this implies that $\tau^1$ is the Denjoy-Wolff point of $\phi_{\mu},$ for all $\mu,$ hence (iii) holds.
\par Finally, to prove the right Type I-version of the theorem, notice that the function $\widetilde{\phi}:\DD^2\to\DD$ defined by $\widetilde{\phi}(\lambda)=\phi(\lambda^2, \lambda^1)$ ($\lambda\in\DD^2$) has $(\widetilde{M}, \widetilde{u})$ as a model, where $\widetilde{u}: \DD^2\to \widetilde{M}=M^2\oplus M^1$ is defined as 
$$\widetilde{u}(\lambda)=\langle \widetilde{u}^1_{\lambda}, \widetilde{u}^2_{\lambda},\rangle=\langle u^2_{(\lambda^2, \lambda^1)}, u^1_{(\lambda^2, \lambda^1)}\rangle,$$
for all $\lambda\in\DD^2.$ By definition, $(\sigma, \tau^2)$ is a right Type I DW point for $\phi$ if and only if $(\tau^2, \sigma)$ is a left Type I DW point for $\widetilde{\phi}.$ Thus, to obtain the right Type I-version of Theorem \ref{TypeIcharabridged}, one simply has to apply the left Type I-version of that same theorem to $\widetilde{\phi}.$
\end{proof}
We also establish a uniqueness result for Type I DW points.
\begin{proposition}\label{TypeIuniqueness}
Let $\phi\in\mathcal{S}_2$ be a left Type I function with model $(M, u)$ such that $\phi\neq\pi^1$ and $\tau^1\in\mathbb{T}$ is the common Denjoy-Wolff point of all maps $\phi_{\mu}.$ Then, there exists $u_{\tau^1}=(u^1_{\tau^1}, 0)\in M$ such that $||u_{\tau^1}||^2=\phi'_{\mu}(\tau^1)\le 1,$ for all $\mu.$ Moreover,
given any $\sigma=(\sigma^1, \sigma^2)\in \mathbb{T}\times\text{cl}(\DD),$ if 
\begin{itemize}
    \item[(i)] $\sigma^1=\tau^1$ and $|\sigma^2|=1,$ we have $u_{\tau^1}\in Y_{\sigma}$. Also, given any $v_{\sigma}\in Y_{\sigma}$, we have $v^2_{\sigma}=0$ if and only if $v_{\sigma}=u_{\tau^1}$. If, in addition, we assume that $\sigma$ is a C-point, we obtain that every $v_{\sigma}\in Y_{\sigma}$ that is not equal to $u_{\tau^1}$ must satisfy $||v^1_{\sigma}||>||u_{\tau^1}||$ and $v^2_{\sigma}\neq 0$;
    \item[(ii)] $\sigma^1=\tau^1$ and $|\sigma^2|<1,$ we have $Y_{\sigma}=\{u_{\tau^1}\};$
    \item[(iii)] $\sigma^1\neq \tau^1$, $\sigma$ is a B-point for $\phi$ and $\phi(\sigma)=\sigma^1,$ then every $v_{\sigma}\in Y_{\sigma}$ must satisfy  either $||v^1_{\sigma}||>1$ or $||v^1_{\sigma}||=1$ and $v^2_{\sigma}\neq 0.$
    
\end{itemize}
Consequently, if $\sigma=(\sigma^1, \sigma^2)\in \mathbb{T}\times\text{cl}(\DD),$ then $\sigma$ is a left Type I DW point for $\phi$ if and only if $\sigma^1=\tau^1.$ Also, no point in $\mathbb{T}\times\text{cl}(\DD)$ can be a left Type II DW point for $\phi$.\par 
There is an analogous statement for right Type I DW points (we need to
assume that $\phi\neq\pi^2$).
\end{proposition}
\begin{proof}
Let $\phi$ be a left Type I function satisfying our assumptions and denote by $u_{\tau^1}\in Y_{(\tau^1, \sigma)}$ (for all $|\sigma|\le 1$)  the vector described after the  ``(i) implies (iii)" part of the proof of Theorem \ref{TypeIcharabridged}. Also, let $\sigma=(\sigma^1, \sigma^2)\in \mathbb{T}\times\text{cl}(\DD).$\par 
First, assume $\sigma^1=\tau^1$ and $|\sigma^2|=1.$ The conclusions of (i) then follow by invoking Lemma \ref{twonullcoords} and Theorem \ref{BbutnotC}.\par 

On the other hand, if  $\sigma^1=\tau^1$ and $|\sigma^2|<1,$ an application of Theorem \ref{facialB} does the job. \par 
Now, assume $\sigma^1\neq \tau^1$, $\sigma$ is a B-point for $\phi$ and $\phi(\sigma)=\sigma^1$. Let $v_{\sigma}\in Y_{\sigma}$ be such that $||v^1_{\sigma}||\le 1$ and choose $\{(\lambda_n, \mu_n)\}\subset \DD^2$ that converges to $\sigma$ and also satisfies $\lim_n\phi(\lambda_n, \mu_n)=\sigma^1$ and $u_{(\lambda_n, \mu_n)}\to v_{\sigma}$ weakly as $n\to\infty.$ Setting $(\lambda, \mu)=(\lambda_n, \mu_n)$ in (\ref{slicemodel}) and letting $n\to\infty$ then allows us to obtain
$$1-\sigma^1\overline{\tau^1}=(1-\sigma^1\overline{\tau^1})\langle v^1_{\sigma}, u^1_{\tau^1}\rangle.$$
Since $\sigma^1\neq \tau^1,$ we obtain 
$$\langle v^1_{\sigma}, u^1_{\tau^1}\rangle=1\ge ||v^1_{\sigma}||^2, ||u^1_{\tau^1}||^2,$$
which implies that $v^1_{\sigma}=u^1_{\tau^1}$ and both have to be unit vectors. However, if we also assume that $v^2_{\sigma}=0,$ we obtain that $\sigma$ is a left Type I DW point for $\phi.$ In view of Theorem \ref{TypeIcharabridged}, this implies that the common Denjoy-Wolff point of all maps $\phi_{\mu}$ is $\sigma^1\neq \tau^1,$ a contradiction (since $\phi\neq\pi^1$). Thus, we must have $v^2_{\sigma}\neq 0$ and the proof of (iii) is complete.
\par 
Finally, to prove the right Type I-version of the theorem, apply the left Type I-version to $\widetilde{\phi}$.
\end{proof}
Note also the following consequence of Theorem \ref{TypeIcharabridged}, which (especially the second part) will be instrumental in Section \ref{refineHerve}.
\begin{corollary}\label{typeIhorosph}
 Let $\phi: \DD^2\to\DD$, $\phi\neq \pi^1$, be holomorphic. Then, $\phi$ has a left Type I DW point of the form $(\tau^1, \sigma)\in\mathbb{T}\times\text{cl}(\DD)$ if and only if 
 $$\frac{|\tau^1-\phi(\lambda, \mu)|^2}{1-|\phi(\lambda, \mu)|^2}\le \frac{|\tau^1-\lambda^2|^2}{1-|\lambda|^2}, \hspace{0.5 cm} \forall (\lambda, \mu)\in\DD^2.$$
  \par 
 If, in addition, we assume that $\tau=(\tau^1, \tau^2)$ is not a C-point for some $\tau^2\in\mathbb{T}$, then for any increasing sequence $\{M_k\}\subset\mathbb{R}^{+}$ tending to $\infty$ one can find a sequence $\{r_k\}$ such that $r_k>1,$ $r_k\to 1$ and 
 $$\frac{|\tau^1-\phi(\lambda, \mu)|^2}{1-|\phi(\lambda, \mu)|^2}\le \max\bigg\{\frac{1}{r_k}\frac{|\tau^1-\lambda|^2}{1-|\lambda|^2}, \frac{1}{M_k}\frac{|\tau^1-\lambda|^2}{1-|\lambda|^2} \bigg\},$$
 for all $\lambda, \mu\in\DD$ and $k\ge 1.$ \par 
 There is an analogous statement for right Type I DW points.
\end{corollary}
\begin{proof} We only prove the left Type I-version. 
    Since $\tau^1$ will be the Denjoy-Wolff point of every map $\phi_{\mu}$, to obtain the first part of the theorem it suffices (in view of Theorem \ref{TypeIcharabridged}) to apply the one-variable Julia's inequality to every $\phi_{\mu}.$ \par 
    To prove the second part, assume that there exists $\tau^2\in\mathbb{T}$ such that $\tau=(\tau^1, \tau^2)$ is not a C-point (it will necessarily be a B-point). In view of Proposition \ref{directdersincreasing} and Theorem \ref{BbutnotC}, $\{K_{\tau}(M_k)\}_k$ will be strictly increasing, hence 
    $$||x^1_{\tau}(\delta_{M_k})||^2+M_k||x^2_{\tau}(\delta_{M_k})||^2<1,$$
    for all $k\ge 1.$ In particular, we can find $r_k>1$ such that 
    $$
K_{\tau}\bigg(\frac{M_k}{r_k}\bigg)\le \frac{1}{r_k}
   $$
    for all $k\ge 1.$ Theorems \ref{generaldir} and \ref{Juliarevisited} then allows us to deduce the desired inequality.
\end{proof}

Next, we turn to Type II DW points. 
\begin{proof}[Proof of Theorem \ref{TypeIIcharabridged}] Let $(M, u)$ be a model for $\phi.$ \par 
First, we show that (i) implies (ii). By assumption, $\tau$ is a B-point for $\phi$ that is not a left Type I DW point, $\phi(\tau)=\tau^1$ and also there exists $u_{\tau}\in Y_{\tau}$ such that $||u^1_{\tau}||< 1$ and 
\begin{equation}\label{useful}
    ||u^1_{\tau}||^2+K||u^2_{\tau}||^2\le 1.
\end{equation}
\par To begin, we show that $\phi$ has to be a left Type II function. Indeed, assume instead that $\phi$ is a left Type I function, $\sigma^1\in\mathbb{T}$ being the common Denjoy-Wolff point of all maps $\phi_{\mu}.$ We cannot have $\sigma^1=\tau^1,$ since then $\tau$ would be (in view of Theorem \ref{TypeIcharabridged}) a left Type I DW point, contradicting the definition of a left Type II DW point. On the other hand, if $\sigma^1\neq \tau^1,$ we obtain a contradiction in view of Proposition \ref{TypeIuniqueness}(iii). Thus, $\phi$ cannot be a left Type I function and we conclude (by Theorem \ref{130}) that $\phi$ is a left Type II function. \par 
Now, let $\xi:\DD\to\DD$ denote the holomorphic function that keeps track of the unique (interior) fixed point of each slice $\phi_{\mu},$ i.e. we have $\phi(\xi(\mu), \mu)=\xi(\mu),$ for all $\mu\in\DD.$ Let $0<K'<K$. Since \ref{useful} holds and $u^2_{\tau}\neq 0$, we must have 
$ r||u^1_{\tau}||^2+K'||u^2_{\tau}||^2\le 1$ whenever $r>1$ is sufficiently close to $1$, hence 
$$ 
   ||u^1_{\tau}||^2+\frac{K'}{r}||u^2_{\tau}||^2\le \frac{1}{r}.
$$ 
Proposition \ref{directdersincreasing} then implies that $$\frac{D_{-(\tau^1, \tau^2 K'/r)}\phi(\tau)}{-\tau^1}=K_{\tau}(K'/r)\le 1/r.$$
In view of Theorem \ref{Juliarevisited}, we obtain 
\begin{equation}\label{useful2}
    \frac{|\tau^1-\phi(\lambda, \mu)|^2}{1-|\phi(\lambda, \mu)|^2}\le \max \bigg\{\frac{1}{r}\frac{|\tau^1-\lambda|^2}{1-|\lambda|^2},  \frac{1}{K'}\frac{|\tau^2-\mu|^2}{1-|\mu|^2}\bigg\},
\end{equation}
for all $\lambda, \mu\in\DD.$ Plugging in $\lambda=\xi(\mu)$ in (\ref{useful2}) then gives us 
$$\frac{|\tau^1-\xi(\mu)|^2}{1-|\xi(\mu)|^2}\le \max \bigg\{\frac{1}{r}\frac{|\tau^1-\xi(\mu)|^2}{1-|\xi(\mu)|^2},  \frac{1}{K'}\frac{|\tau^2-\mu|^2}{1-|\mu|^2}\bigg\}, $$
for all $\mu\in\DD.$ Since $1/r<1,$ this last inequality implies 
$$\frac{1}{r}\frac{|\tau^1-\xi(\mu)|^2}{1-|\xi(\mu)|^2}\le  \frac{1}{K'}\frac{|\tau^2-\mu|^2}{1-|\mu|^2}$$
whenever $r>1$ is sufficiently close to $1.$ Letting $r\to 1$ first and $K'\to K$ afterwards yields
$$\frac{|\tau^1-\xi(\mu)|^2}{1-|\xi(\mu)|^2}\le  \frac{1}{K}\frac{|\tau^2-\mu|^2}{1-|\mu|^2},$$
for all $\mu\in\DD.$
The one-variable Julia's inequality (see Section \ref{prelims}) then allows us to deduce that $\tau^2$ is a B-point for $\xi$, $\xi(\tau^2)=\tau^1$ and also 
\begin{equation}\label{Adef}
  A:=\bigg(\liminf_{\mu\to\tau^1}\frac{1-|\xi(\mu)|}{1-|\mu|}\bigg)^{-1}\ge K.  
\end{equation}
\par 
To show that (ii) implies (i), assume that $\phi$ is a left Type II function and $\xi$ satisfies the given hypotheses. Substituting $\lambda=\xi(\mu)$ into the model formula 
$$1-|\phi(\lambda, \mu)|^2=(1-|\lambda|^2)||u^1_{(\lambda, \mu)}||^2+(1-|\mu|^2)||u^2_{(\lambda, \mu)}||^2 $$
yields 
$$\frac{1-|\phi(\xi(\mu), \mu)|^2}{1-|\mu|^2}= \frac{1-|\xi(\mu)|^2}{1-|\mu|^2} $$
\begin{equation}\label{useful5}
 =\frac{1-|\xi(\mu)|^2}{1-|\mu|^2}||u^1_{(\xi(\mu), \mu)}||^2+||u^2_{(\xi(\mu), \mu)}||^2,
\end{equation}
for all $\mu\in\DD.$ By assumption, we can find a (radial) sequence $\{\mu_n\}\subset\DD$ such that $\lim_n\mu_n=\tau^2$, $\lim_{n}\xi(\mu_n)=\lim_n\phi(\xi(\mu_n), \mu_n)=\tau^1$ and 
$$\lim_n\frac{1-|\xi(\mu_n)|}{1-|\mu_n|}=\lim_n\frac{1-|\xi(\mu_n)|^2}{1-|\mu_n|^2}\le \frac{1}{K}.$$
Note also that $\lim_n\frac{1-|\xi(\mu_n)|}{1-|\mu_n|}>0,$ else the single-variable Julia's inequality would imply that $\xi$ is a unimodular constant, a contradiction. Thus, plugging in $\mu=\mu_n$ in (\ref{useful5}) and letting $n\to\infty$ allows us to conclude that $\tau$ is a B-point for $\phi,$ $\phi(\tau)=\tau^1$ and also there exists $u_{\tau}\in Y_{\tau}$ such that 
$$||u^1_{\tau}||^2+K||u^2_{\tau}||^2\le 1.$$
Moreover, since $\phi$ is a left Type II function, Theorem \ref{TypeIcharabridged} implies that $\tau$ cannot be a left Type I DW point and $u^2_{\tau}\neq 0,$ hence $||u^1_{\tau}||<1$ and we are done. \par 
Note that the previous argument actually shows that $A$ (as defined in (\ref{Adef})) is the maximum among all constants $K>0$ such that $\tau$ is a left Type II DW point for $\phi$ with constant $K$.
\par 
Next, we show that (i) implies (iii). So, assume that all relevant assumptions are satisfied. Note that we cannot have $$\frac{D_{-(\tau^1, \tau^2 M)}\phi(\tau)}{-\tau^1}=K_{\tau}(M)\le 1$$ for all $M>0,$ as in such a case Theorem \ref{TypeIcharabridged}) would imply that $\tau$ is a left Type I DW point, a contradiction. Since $K_{\tau}(M)$ is continuous, increasing and $K_{\tau}(K)\le 1$, there must exist $C\ge K$ such that $K_{\tau}(C)=1$. Moreover, $K_{\tau}(M)$ cannot be constant (again by Theorem \ref{TypeIcharabridged}), hence (iii) holds. 
\par 
We now prove the converse. Assume $\tau$ is a B-point for $\phi$, $\phi(\tau)=\tau^1,$ $K_{\tau}(M)$ is not constant with respect to $M$ and also there exists $C\ge K$ such that $K_{\tau}(C)=1$, hence 
$$||x^1_{\tau}(\delta_C)||^2+C||x^2_{\tau}(\delta_C)||^2=1.$$
We cannot have $x^2_{\tau}(\delta_C)=0$ (else, Theorem \ref{nullcompgivesCpoint} would imply that $K_{\tau}(M)$ is constant), thus $||x^1_{\tau}(\delta_C)||<1.$ Moreover, $\tau$ cannot be a left Type I DW point, as, in view of Theorem \ref{TypeIcharabridged} and the equality  $K_{\tau}(C)=1$, the only way for this to be possible would be having $K_{\tau}(M)=1,$ for all $M>0,$ a contradiction. Thus, $\tau$ is a left Type II DW point with constant $C\ge K$ and we are done. \par 
We can say more about the constant $C$ (which is uniquely determined, as $K_{\tau}(M)$ is strictly increasing). Indeed, our previous argument shows that $\tau$ is a left Type II DW point with constant $C.$ Now, if $C'>C,$ then 
$$1=K_{\tau}(C)<K_{\tau}(C'),$$
and thus, in view of ``(i) implies (iii)", we obtain that $\tau$ cannot be a left Type II DW point with costant $C'.$ This means that $C$ is the largest constant with this property, hence $C=A,$ as defined in (\ref{Adef}).
\par 
Finally, as seen in the end of the proof of Theorem \ref{TypeIcharabridged}, to show the right Type II-version of the theorem we only need apply the left Type II-version to $\widetilde{\phi}$.
\end{proof}
The proof of Theorem \ref{ultimateDWchar} now follows easily by combining all of our previous results. 
\begin{proof}[Proof of Theorem \ref{ultimateDWchar}]
Combine Theorems \ref{TypeIcharabridged}-\ref{TypeIIcharabridged} with Lemma \ref{facialBisC}, Theorem \ref{nullcompgivesCpoint} and Proposition \ref{directdersincreasing}.
\end{proof}
\begin{remark}
Let $\xi: \DD\to\DD$ be holomorphic. Then, one can always find $\phi\in\mathcal{S}_2$ (that will necessarily be a left Type II function) such that $\phi(\xi(\mu), \mu)=\xi(\mu)$ for all $\mu\in\DD.$ Indeed, it can be easily verified that the function 
$$\phi(\lambda, \mu):=\frac{\lambda+\xi(\mu)}{2}$$
has the property in question.
\end{remark}
\begin{remark}\label{TypeIIwithout}
As already mentioned in Section \ref{prelims}, there exist left Type II functions that do not have left Type II DW points. Indeed, if e.g. $\phi$ is any left Type II function such that the map $\xi$ satisfies $\xi(\DD)\subset r\DD$ for some $r\in (0, 1),$ then Theorem \ref{TypeIIcharabridged} implies that $\phi$ does not have any left Type II DW points (on account of $\xi$ not having any B-points). 
\end{remark}
We can also prove certain uniqueness results for Type II DW points. 
\begin{proposition}\label{TypeIIuniqueness}
Let $\phi:\DD^2\to\DD$ with model $(M, u)$ be such that $\tau=(\tau^1, \tau^2)\in\mathbb{T}^2$ is a left Type II DW point, with $\xi:\DD\to \DD$   satisfying $\phi(\xi(\mu), \mu)=\xi(\mu),$ for all $\mu\in\DD,$ and $A>0$ defined as in (\ref{Adef}). Then, the following assertions all hold.
\begin{itemize}
    \item[(i)] $x_{\tau}(\delta_A)$ is the unique vector $u_{\tau}\in Y_{\tau}$ such that \begin{equation}\label{useful7}
        ||u^1_{\tau}||^2+A||u^2_{\tau}||^2\le 1.
    \end{equation}
     \item[(ii)] No point in $\mathbb{T}\times \text{cl}(\DD)$ can be a left Type I DW point for $\phi.$ 
    \item[(iii)] If $\sigma\in\mathbb{T}$ and $\sigma\neq \tau^1,$ then $(\sigma, \tau^2)$ is not a left Type II DW point for $\phi$. 
   
\end{itemize}
\par 
There is an analogous result for right Type II DW points.
\end{proposition}
\begin{proof} First, we prove (i). Note that $x_{\tau}(\delta_A)$ certainly satifies 
$$||x^1_{\tau}(\delta_A)||^2+A||x^2_{\tau}(\delta_A)||^2=1,$$
as $K_{\tau}(A)=1.$ Also, if $u_{\tau}\in Y_{\tau}$ is such that  (\ref{useful7}) holds, Proposition \ref{directdersincreasing} implies that $||u^1_{\tau}||^2+A||u^2_{\tau}||^2=1$ and $x_{\tau}(\delta_A)=u_{\tau},$ as desired. \par 
(ii) is an immediate consequence of Proposition \ref{TypeIuniqueness}. 
\par 
Finally, (iii) is a simple application of Theorem \ref{TypeIIcharabridged}, since $\xi$ cannot have two distinct values (at least not in the sense of nontangential limits) at its B-point $\tau^2.$ 
\end{proof}
 The following Julia-type inequalities are obtained as a consequence of Theorem \ref{TypeIIcharabridged}. The significance of parts (ii) and (iii) will be made apparent in Section \ref{refineHerve}.
\begin{corollary}\label{TypeIIhorosph}
Assume $\phi:\DD^2\to\DD$ has a left Type II DW point $\tau=(\tau^1, \tau^2)\in\mathbb{T}^2$ and let $A>0$ be defined as in (\ref{Adef}). Also, fix $A_{-}<A$ and $r_1<1.$
\begin{itemize}
    \item[(i)] For all $(\lambda, \mu)\in\DD^2$, we have
$$\frac{|\tau^1-\phi(\lambda, \mu)|^2}{1-|\phi(\lambda, \mu)|^2}\le \max\bigg\{\frac{|\tau^1-\lambda|^2}{1-|\lambda|^2}, \frac{1}{A} \frac{|\tau^2-\mu|^2}{1-|\mu|^2} \bigg\};$$ 
    \item[(ii)] Moreover, if $r_2>1$ is sufficiently close to 1, then 
$$\frac{|\tau^1-\phi(\lambda, \mu)|^2}{1-|\phi(\lambda, \mu)|^2}\le \max\bigg\{\frac{1}{r_2}\frac{|\tau^1-\lambda|^2}{1-|\lambda|^2}, \frac{1}{A_{-}} \frac{|\tau^2-\mu|^2}{1-|\mu|^2} \bigg\},$$
for all $(\lambda, \mu)\in\DD^2$;
    \item[(iii)] Finally,  if $x^1_{\tau}(\delta_A)\neq 0$ and $A<A_{+}$ is sufficiently close to $A$, then
    $$\frac{|\tau^1-\phi(\lambda, \mu)|^2}{1-|\phi(\lambda, \mu)|^2}\le \max\bigg\{\frac{1}{r_1}\frac{|\tau^1-\lambda|^2}{1-|\lambda|^2}, \frac{1}{A_{+}} \frac{|\tau^2-\mu|^2}{1-|\mu|^2} \bigg\}, $$
    for all $(\lambda, \mu)\in\DD^2.$
\end{itemize}\par 
There is an analogous result for right Type II DW points. 
\end{corollary}
\begin{proof}
To prove (i), combine Theorems \ref{Juliarevisited} and \ref{TypeIIcharabridged}.\par 
For (ii), note that, since $||x^1_{\tau}(\delta_A)||^2+A||x^2_{\tau}(\delta_A)||^2=1$, $A_{-}<A$ and $x^2_{\tau}(\delta_A)\neq 0$ 
 (by definition of a left Type II DW point), one obtains that 
 $$r_2||x^1_{\tau}(\delta_A)||^2+A_{-}||x^2_{\tau}(\delta_A)||^2\le 1,$$
 for all $r_2>1$ sufficiently close to $1,$ hence $K_{\tau}(A_{-}/r_2)\le 1/r_2.$ An application of Theorem \ref{Juliarevisited} then finishes the job. \par 
 (iii) is proved in an analogous manner (note that we have to assume $x^1_{\tau}(\delta_A)\neq 0$, since not all left Type II DW points have this property). 
 \end{proof}

 \small
\section{REFINING HERV\'{E}'S THEOREM} \label{refineHerve}
\large 
Let $F=(\phi, \psi)$ denote a holomorphic self-map of $\DD^2$ without interior fixed points. We use 
$$F^n=(\phi_n, \psi_n)=\underbrace{F\circ F\circ \cdots \circ F}_\textrm{$n$ times}$$
to denote the sequence of iterates of $F.$ Note that $\phi\circ F^n=\phi_{n+1}$ and $\psi\circ F^n=\psi_{n+1},$ for all $n\ge 1.$ \par 
As already mentioned in Section \ref{intro}, Herv\'{e} analysed the behavior of $\{F^n\}$ by looking at three separate cases, depending on the Type of $\phi$ and $\psi.$ In this section, we study the connection between  Herv\'{e}'s results and the DW points we defined in Section \ref{charactDW}. In particular, we will show how the conclusions of Theorem \ref{HERVE's THEOREM} can be strengthened if one assumes that the DW points of $\phi$ and/or $\psi$ are not C-points.

 \small
\subsection{The (Type II, Type II) case} 
\large

We begin with the case where $\phi$ and $\psi$ are left Type II and right Type II functions, respectively.  Even though not every Type II function will, in general, have Type II DW points (see Remark \ref{TypeIIwithout}), $F$ having no interior fixed points changes the situation dramatically, as seen in the following theorem. A proof of it (without the model terminology) is essentially contained in \cite[Theorem 2]{FrosiniArxiv} (see also \cite[Section 16]{HERVE}). We give an alternative proof by using the results we have developed so far. 
\begin{theorem}\label{(Type II, Type II)}
    Assume $F=(\phi, \psi):\DD^2\to\DD^2$ is holomorphic and $\phi, \psi$ are left Type II and right Type II functions, respectively. Also, let $\xi, \eta:\DD\to\DD$ denote the (unique) functions such that 
    $\phi(\xi(\mu), \mu)=\xi(\mu)$ and $\psi(\lambda, \eta(\lambda))=\eta(\lambda),$ for all $\lambda, \mu\in\DD.$ Then, $F$ has no interior fixed points if and only if \begin{itemize}
        \item[(i)] there exist $\tau\in\mathbb{T}^2$ and $K>0$ such that $\tau$ is simultaneously a left Type II DW point for $\phi$ with constant $K$ and a right Type II DW point for $\psi$ with constant $1/K$
   and also
        \item[(ii)] $\phi\circ \eta \neq\text{Id}_{\DD}$ and $\eta\circ\phi\neq\text{Id}_{\DD}$.
    \end{itemize}
    Moreover, assuming $F$ has no interior fixed points, the point  $\tau=(\tau^1, \tau^2)\in\mathbb{T}^2$ above is uniquely determined: $\tau^1$ is the Denjoy-Wolff point of $\xi\circ\eta,$ while $\tau^2$  is the Denjoy-Wolff point of $\eta\circ\xi.$ 
\end{theorem}
\begin{proof} Let $(M, u), (N, v)$ be models for $\phi$ and $\psi$, respectively. Also, for $\tau\in\partial\DD^2,$ we will denote the corresponding cluster sets by $Y^{\phi}_{\tau}$ and $Y^{\psi}_{\tau}.$
\par
First, assume $F$ has no interior fixed points. Let $0<r_n\uparrow 1$ and consider the functions $r_n\cdot F$. Since $\text{cl}(r_nF(\DD^2))\subset \DD^2$, for every $n$, the Earle-Hamilton Theorem \cite{EarleHam} implies that each $r_nF$ has a fixed point $(\lambda_n, \mu_n)\in\DD^2$. Since $F$ has no fixed points in $\DD^2,$ we obtain that  $(\lambda_n, \mu_n)\to\partial\DD^2$. There are three possible cases to examine. \par 
 If $\lim_n\frac{1-|\lambda_n|^2}{1-|\mu_n|^2}=0$, then $(\lambda_n, \mu_n)\to  \tau=(\tau^1, \sigma)\in \mathbb{T}\times\text{cl}(\DD)$. We can use the model formula for $\phi$ to write  $$1-|\lambda_n|^2\ge 1-\frac{1}{r^2_n}|\lambda_n|^2=1-|\phi(\lambda_n, \mu_n)|^2$$ $$=(1-|\lambda_n|^2)||u^1_{(\lambda_n, \mu_n)}||^2+(1-|\mu_n|^2)||u^2_{(\lambda_n, \mu_n)}||^2.$$
Thus, for $n$ large enough, we deduce 
$$ 1\ge \frac{1-|\phi(\lambda_n, \mu_n)|^2}{1-||(\lambda_n, \mu_n)||^2}=\frac{1-|\phi(\lambda_n, \mu_n)|^2}{1-|\lambda_n|^2} $$ \begin{equation} \label{useful13}=||u^1_{(\lambda_n, \mu_n)}||^2+\frac{1-|\mu_n|^2}{1-|\lambda_n|^2}||u^2_{(\lambda_n, \mu_n)}||^2.  
\end{equation}
Letting $n\to\infty$, we obtain (in view of $\lim_n\frac{1-|\lambda_n|^2}{1-|\mu_n|^2}=0$ and $\lim_n \phi(\lambda_n, \mu_n)=\tau^1$) that $\tau=(\tau^1, \sigma)$ is a B-point for $\phi,$ $\phi(\tau)=\tau^1$ and also there exists a weak limit $u_{\tau}\in Y^{\phi}_{\tau}$ such that $||u^1_{\tau}||\le 1, u^2_{\tau}=0$. This implies that $\tau$ is a left Type I DW point for $\phi$, contradicting the fact that $\phi$ is a left Type II function. \par 
  If $\lim_n\frac{1-|\lambda_n|^2}{1-|\mu_n|^2}=\infty$, one can argue in a manner analogous to the previous case to deduce that $\psi$ has a right Type I DW point, which is again a contradiction. \par 
Finally, assume that $\lim_n\frac{1-|\lambda_n|^2}{1-|\mu_n|^2}=\frac{1}{K}\in (0, \infty)$. Hence, $(\lambda_n, \mu_n)\to  \tau=(\tau^1, \tau^2)\in \mathbb{T}^2$. Letting $n\to\infty$ in (\ref{useful13}) then yields that  $\tau$ is a $B$-point for $\phi$, $\phi(\tau)=\tau^1$ and also there exists $u_{\tau}\in Y^{\phi}_{\tau}$ such that $||u^1_{\tau}||^2+K||u^2_{\tau}||^2\le 1.$ Note that $u^2_{\tau}\neq 0,$ else $\tau$ would be a left Type I DW point. Thus, since $\phi$ is a left Type II function, $\tau$ must be a left Type II DW point for $\phi$ with constant $K$. Further, an analogous argument involving the model formula for $\psi$ shows that $\tau$ is a $B$-point for $\psi$, $\phi(\tau)=\tau^2$ and also there exists $v_{\tau}\in Y^{\psi}_{\tau}$ such that $(1/K)||v^1_{\tau}||^2+||v^2_{\tau}||^2\le 1.$ Also, $v^1_{\tau}\neq 0,$ since $\psi$ is not a right Type I function. Thus, $\tau$ must be a right Type II DW point for $\psi$ with constant $1/K,$ which proves (i). To show that (ii) holds, note that if e.g. $\xi(\eta(\lambda))=\lambda$ for some $\lambda\in\DD,$ then $$F(\xi(\eta(\lambda)), \eta(\lambda))=\big(\phi(\xi(\eta(\lambda)), \eta(\lambda)), \psi(\xi(\eta(\lambda)), \eta(\lambda)) \big)$$
$$=(\xi(\eta(\lambda)),  \eta(\lambda)),$$
a contradiction. In particular, we obtain the even stronger conclusion that neither $\xi\circ\eta$ nor $\eta\circ\xi$ can have interior fixed points. 
\par 
Conversely, assume that (i) and (ii) both hold. In view of Theorem \ref{TypeIIcharabridged}, (i) implies that $\tau^1$ and $\tau^2$ are B-points for $\eta$ and $\xi$ respectively, $\xi(\tau^2)=\tau^1,$  $\eta(\tau^1)=\tau^2$ and also (by the single-variable Julia's inequality)
$$\xi(E(\tau^2, R))\subset E(\tau^1, R/K) \hspace{0.3 cm}\text{ and }  \hspace{0.3 cm}  \eta(E(\tau^1, R))\subset E(\tau^2, KR),$$
for all $R>0.$ Thus, $(\xi\circ\eta)(E(\tau^1, R))\subset \xi(E(\tau^2, KR))\subset E(\tau^1, R)$, for all $R>0,$ which (combined with the fact that $\xi\circ\eta\neq \text{Id}_{\DD}$ must have a unique Denjoy-Wolff point) allows us to deduce that $\tau^1$ is the Denjoy-Wolff point of $\xi\circ\eta$. An analogous argument shows that $\tau^2$ is the Denjoy-Wolff point of $\eta\circ\xi.$ Thus, the point $\tau$ is indeed uniquely determined. Also, notice that, in view of these observations,  neither $\xi\circ\eta$ nor $\eta\circ\xi$ can have interior fixed points. Now, let $(\lambda_0, \mu_0)$ be an interior fixed point of $F$. We obtain 
$$\phi(\lambda_0, \mu_0)=\lambda_0 \hspace{0.3 cm}\text{ and }  \hspace{0.3 cm}  \psi(\lambda_0, \mu_0)=\mu_0.$$
Thus, $\xi(\mu_0)=\lambda_0$ and $\eta(\lambda_0)=\mu_0,$ which implies that $\xi(\eta(\lambda_0))=\lambda_0$, a contradiction. 
\end{proof}
Now, let $F=(\phi, \psi):\DD^2\to\DD^2$,  $\tau\in\mathbb{T}^2$ and $K>0$ be as in Theorem \ref{(Type II, Type II)}, with $F$ having no interior fixed points. Recall that, in this setting, one obtains a perfect analogue of the one-variable Denjoy-Wolff Theorem, i.e. the sequence of iterates $\{F^n\}$ converges uniformly on compact sets to $\tau$ (Theorem \ref{HERVE's THEOREM}(iv)). A crucial ingredient for Herv\'{e}'s proof of this fact is given by the invariant horospheres
\begin{equation}\label{useful20}
F(E(\tau, R, KR)) \subset E(\tau, R, KR),
\end{equation}
obtained as an application of Corollary \ref{TypeIIhorosph}.
\par So, we know that the entire sequence $\{F^n\}$ has to converge to $\tau$, but can we use (\ref{useful20}) to say more?  Our main result in this subsection is a refinement of \cite[Lemme 2]{HERVE}, which concerns the location of the orbits $\{F^n(\lambda, \mu)\}_n$ with respect to the boundary of the invariant horospheres (\ref{useful20}). To set up the statement, fix $(\lambda_0, \mu_0)\in\DD^2$. For convenience, we will write $F^n=(\phi_n, \psi_n)$ in place of $F^n(\lambda_0, \mu_0)=(\phi_n(\lambda_0, \mu_0), \psi_n(\lambda_0, \mu_0))$. We also define:
  $$A_n=\frac{|\tau^1-\phi_n|^2}{1-|\phi_n|^2} \hspace{0.3 cm} \text{ and }\hspace{0.3 cm} B_n=\frac{|\tau^2-\psi_n|^2}{1-|\psi_n|^2}.$$

\begin{theorem}\label{motionin(II,II)}
Let $F=(\phi, \psi):\DD^2\to\DD^2$, $\tau\in\mathbb{T}^2$ and $K>0$ be as in Theorem \ref{(Type II, Type II)}, with $F$ having no interior fixed points. Then, either $F^n\to\tau$ in the horospheric topology or there exist $\rho_0, \rho_1\ge 0$ (depending on $(\lambda_0, \mu_0)$) that are not both $0$ such that  $$A_{2n}\to \rho_0, \hspace{0.4 cm} A_{2n+1}\to \rho_1,  \hspace{0.4 cm} B_{2n+1}\to K\rho_0, \hspace{0.4 cm} B_{2n}\to K\rho_1.$$
Moreover, if $\tau$ is not a C-point for either $\phi$ or $\psi$, we can take $\rho_0=\rho_1.$
\end{theorem}
\begin{proof}
For every $n\ge 1$, let $R_n$ denote the smallest radius such that $F_n\in E_n:=\text{cl}\big(E(\tau, R_n, KR_n)\big)$. In view of (\ref{useful20}), the sequence $\{R_n\}$ is non-increasing. $\{A_n\}, \{B_n\}$ needn't also be non-increasing, however they  have to satisfy (by definition of $R_n$) $\max\{KA_n, B_n\}=KR_n$, for all $n$. \par 
Now, if $R_n\to 0,$ then $A_n, B_n\to 0$  and we conclude that $F_n\to \tau$ in the horospheric topology. So, assume $R_n$ converges to $\rho>0$. \par 
First, consider the case where $\tau$ is not a C-point for either $\phi$ or $\psi.$
Without loss of generality, we may suppose that $\tau$ is not a C-point for $\phi.$ Let $u_{\tau}$ denote any vector in $Y^{\phi}_{\tau}$ such that $||u^1_{\tau}||^2+K||u^2_{\tau}||^2\le 1$ (its existence is guaranteed by Theorem \ref{TypeIIcharabridged}). In view of Theorem \ref{nullcompgivesCpoint}, it must be true that $u^1_{\tau}\neq 0.$ We will show that $A_n\to\rho$ and $B_n\to K\rho.$ \par 
Indeed, aiming towards a contradiction, assume $B_n\not\to K\rho$ (the case where $A_n\not\to \rho$ can be treated in an analogous manner). In view of the equality $\max\{KA_n, B_n\}=KR_n$, there exists a subsequence $\{n_k\}$ and $r\in (0, \rho)$ such that $B_{n_k}\le Kr$ for all $k$. This implies that $A_{n_k}=R_{n_k}$ for all $k$. Now, given $0<K_{-}<K$ sufficiently close to $K$,  we can choose $r_2>1$ sufficiently close to $1$ such that $Kr/K_{-}<\rho/r_2$ and also, in view of Corollary \ref{TypeIIhorosph}(ii), 
$$  
 \frac{|\tau^1-\phi(\lambda, \mu)|^2}{1-|\phi(\lambda, \mu)|^2}\le \max \bigg\{\frac{1}{r_2}\frac{|\tau^1-\lambda|^2}{1-|\lambda|^2},\frac{1}{K_{-}}\frac{|\tau^2-\mu|^2}{1-|\mu|^2} \bigg\}, 
$$ 
for all $\lambda, \mu\in\DD.$ In particular, we have 
$$A_{n_k+1}=\frac{|\tau^1-\phi_{n_k+1}|^2}{1-|\phi_{n_k+1}|^2}$$ 
$$\le\max\bigg\{\frac{1}{r_2}A_{n_k}, \frac{1}{K_{-}} B_{n_k}\bigg\} $$
$$\le \max\bigg\{\frac{1}{r_2}R_{n_k}, \frac{1}{K_{-}} Kr \bigg\} $$
\begin{equation}\label{useful50}
 =\frac{R_{n_k}}{r_2},   
\end{equation}
as  $\frac{Kr}{K_{-}}<\frac{\rho}{r_2}\le \frac{R_{n_k}}{r_2}$, for all $k.$ Now, let $v_{\tau}$ denote any vector in $Y^{\psi}_{\tau}$ such that $\tilde{K}||v^1_{\tau}||^2+||v^2_{\tau}||^2\le 1$, where $\tilde{K}=1/K$ (as in the case of $u_{\tau}$, we obtain the existence of this vector by Theorem \ref{TypeIIcharabridged}). We look at two separate cases, depending on whether $v^2_{\tau}\neq 0$. \par 
So, assume $v^2_{\tau}\neq 0$. In this case, given $r_1<1$ sufficiently close to $1,$ we can find $\tilde{K}<\tilde{K}_{+}$ sufficiently close to $\tilde{K}$ such that $\frac{\tilde{K}_{+}r}{r_1}<\tilde{K}\rho$ and also, in view of the right Type II version of Corollary \ref{TypeIIhorosph}(iii), 
$$  
 \frac{|\tau^2-\psi(\lambda, \mu)|^2}{1-|\psi(\lambda, \mu)|^2}\le \max \bigg\{\frac{1}{\tilde{K}_{+}}\frac{|\tau^1-\lambda|^2}{1-|\lambda|^2},\frac{1}{r_1}\frac{|\tau^2-\mu|^2}{1-|\mu|^2} \bigg\}, 
$$ for all $\lambda, \mu\in\DD.$ In particular, we have 
$$B_{n_k+1}\le \max \bigg\{\frac{A_{n_k}}{\tilde{K}_{+}}, \frac{B_{n_k}}{r_1} \bigg\} $$
$$\le \max \bigg\{\frac{R_{n_k}}{\tilde{K}_{+}}, \frac{r}{\tilde{K}r_1} \bigg\}$$ 
\begin{equation}\label{useful60}
   =\frac{R_{n_k}}{\tilde{K}_{+}}, 
\end{equation}
as $\frac{r}{\tilde{K}r_1}<\frac{\rho}{\tilde{K}_{+}}\le \frac{R_{n_k}}{\tilde{K}_{+}}$, for all $k$. Combining (\ref{useful50}) with (\ref{useful60}), we obtain
$$KR_{n_k+1}=\max\{KA_{n_k+1}, B_{n_k+1}\}<cKR_{n_k},$$ for some $c\in (0,1)$ and all $k$ large enough. Letting $k\to\infty$ then leads to a contradiction. \par 
Now, assume $v^2_{\tau}=0$. In view of (\ref{useful50}), we can find $r'<\rho$ such that for all $k$ large enough we have $A_{n_k+1}\le r'<\rho$. Also, since $v^1_{\tau}\neq 0$, we can mimic the proof of (\ref{useful50}) (with $\psi$ in place of $\phi$) to obtain $B_{n_k+2}< c_1KR_{n_k+1}$ for some $c_1\in (0, 1)$ and all $k$ large enough. Similarly, since $u^1_{\tau}\neq 0$, we can mimic the proof of (\ref{useful60}) (with $\phi$ in place of $\psi$) to obtain the existence of $c_2\in (0, 1)$ such that $A_{n_k+2}\le c_2 R_{n_k+1},$ for all $k$ large enough. Thus, we arrive at the conclusion $KR_{n_k+2}=\max\{KA_{n_k+2}, B_{n_k+2}\}<\max\{c_1, c_2\}KR_{n_k+1,}$ for all $k$ large enough, which yields a contradiction when we let $k\to\infty.$ \par 
The only case left to examine is when $R_n\to \rho>0$ and $u^1_{\tau}=v^2_{\tau}=0.$  Mimicking the proof of ``(ii) implies (iii)" from Theorem \ref{Juliarevisited}, we may conclude that
 $$A_{n+1}\le \frac{B_{n}}{K} \hspace{0.1 cm} \text{ and } \hspace{0.1 cm} B_{n+1}\le KA_n,$$
 for all $n\ge 1.$
Thus, $$A_{n+2}\le A_n \hspace{0.1 cm} \text{ and } \hspace{0.1 cm} B_{n+2}\le B_n,$$
which means that the sequences $\{A_{2n}\}, \{A_{2n+1}\}, \{B_{2n}\} $ and $\{B_{2n+1}\}$ are all non-increasing. Thus, there exist nonnegative numbers $\rho_0, \rho_1, \rho'_0, \rho'_1$ such that $A_{2n}\to \rho_0, A_{2n+1}\to \rho_1$, $B_{2n+1}\to \rho'_1$ and $B_{2n}\to \rho'_0.$ The inequalities $A_{2n+1}\le \frac{B_{2n}}{K} $ and $B_{2n}\le KA_{2n-1}$ give us $\rho_1\le \rho'_0/K$ and $\rho'_0\le K\rho_1$, respectively. Thus, $\rho'_0=K\rho_1$ and an entirely analogous argument shows that $\rho'_1=K\rho_0.$ We conclude that 
$$A_{2n}\to \rho_0, \hspace{0.4 cm} A_{2n+1}\to \rho_1,  \hspace{0.4 cm} B_{2n+1}\to K\rho_0, \hspace{0.4 cm} B_{2n}\to K\rho_1,$$
where $\max\{\rho_0, \rho_1\}=\rho$ (by definition of $\rho$) and so $\rho_0, \rho_1$ cannot be zero at the same time. This concludes the proof. \end{proof}

 \small
\subsection{The (Type I, Type II) case} 
\large

Assume now that $\phi$ and $\psi$ are left Type I and right Type II functions, respectively. This immediately implies that $F=(\phi, \psi)$ does not have any interior fixed points. In this setting, Herv\'{e} proved that any cluster point of the sequence of iterates $\{F^n\}$ must be of the form $(\tau^1, h),$ where $h$ is either a holomorphic function $\DD^2\to\DD$ or a unimodular constant and $\tau^1$ is the common Denjoy-Wolff point of all slices $\phi_{\mu}$ (Theorem \ref{HERVE's THEOREM}(iii)).  Examples showing that this conclusion cannot, in general, be improved, are contained in \cite[Section 11]{HERVE}.
\par  Now, if we, in addition, assume the existence of $\sigma\in \mathbb{T}$ such that $(\tau^1, \sigma)$ is a right Type II DW point for $\psi$, stronger conclusions can be drawn about the cluster set of $\{F^n\}.$
\begin{proposition}\label{TypeI,TypeIIWITHDW} Assume $F=(\phi, \psi):\DD^2\to\DD^2$ is such that $\phi$ is a left Type I function (with $\tau^1$ being the common Denjoy-Wolff point of all slices $\phi_{\mu}$) and $\psi$ has a right Type II DW point of the form $\tau=(\tau^1, \sigma)\in\mathbb{T}^2$. Then, there exists $K>0$ such that 
$$F(E(\tau, R, KR))\subset E(\tau, R, KR),$$
for all $R>0.$ Thus, any cluster point of the sequence of iterates $\{F^n\}$ must be of the form $(\tau^1, h),$ where $h$ is either a holomorphic function $\DD^2\to\DD$ or the constant $\sigma.$
\end{proposition}\begin{proof} 
Assuming $\psi$ has a right Type II DW point of the form $\tau=(\tau^1, \sigma)$, one can combine Corollary \ref{typeIhorosph} with Corollary \ref{TypeIIhorosph} to conclude that 
$$F(E(\tau, AR, R))\subset E(\tau, AR, R),$$
for all $R>0$, where $A=\Big(\liminf_{\lambda\to\tau^1}\frac{1-|\eta(\lambda)}{1-|\lambda|}\Big)^{-1}>0$ and $\eta:\DD\to\DD$ is the holomorphic function satisfying $\psi(\lambda, \eta(\lambda))=\eta(\lambda)$ for all $\lambda\in\DD$. 
\par 
To obtain the conclusion regarding the behavior of the iterates, combine the previous result with Theorem \ref{HERVE's THEOREM}(iii) and the observation that, for any $R>0$, cl$(E(\tau, AR, R))\cap\mathbb{T}^2=\tau.$
\end{proof}
\begin{remark}
In the absence of a right Type II DW of the form $(\tau^1, \sigma)$ for $\psi$, the behavior of $\{F^n\}$ could be considerably more complicated. Indeed, it could even happen that infinitely many unimodular constants $\{\sigma(i)\text{ }|\text{ }i\in I\}$ exist such that the constant $(\tau^1, \sigma(i))$ is a cluster point of $\{F^n\}$, for every $i\in I$; see the 2nd example in \cite[Section 11]{HERVE}.
\end{remark} 
In the setting of Proposition \ref{TypeI,TypeIIWITHDW}, it is clear (in view of Theorem \ref{TypeIcharabridged}) that $(\tau^1, \sigma)$ will always be a left Type I DW point for $\phi,$ no matter the value of $\sigma.$ Surprisingly, having $(\tau^1, \sigma)$ not be a C-point for $\phi$ will force the entire sequence  $\{F^n\}$ to converge to $(\tau^1, \sigma).$ This is the content of Theorem \ref{TypeITypeIIBBUTNOTC}, the proof of which does not make use of Herv\'{e}'s results.
\begin{proof}[Proof of Theorem \ref{TypeITypeIIBBUTNOTC}]
  Assume $\tau=(\tau^1, \sigma)\in\mathbb{T}^2$  satisfies the hypotheses of the theorem. Clearly, $\phi$ and $\psi$ will be left Type I and right Type II functions, respectively, with the common Denjoy-Wolff point of all slices $\phi_{\mu}$ being $\tau^1.$ By Proposition \ref{TypeI,TypeIIWITHDW}, there exists $K>0$ such that
  \begin{equation}\label{useful100}
F(E(\tau, KR, R)) \subset E(\tau, KR, R),
\end{equation}
for all $R>0.$  Now, fix $(\lambda_0, \mu_0)\in\DD^2$. For convenience, we will write $F^n=(\phi_n, \psi_n)$ in place of $F^n(\lambda_0, \mu_0)=(\phi_n(\lambda_0, \mu_0), \psi_n(\lambda_0, \mu_0))$. We also define:
  $$A_n=\frac{|\tau^1-\phi_n|^2}{1-|\phi_n|^2} \hspace{0.3 cm} \text{ and }\hspace{0.3 cm} B_n=\frac{|\sigma-\psi_n|^2}{1-|\psi_n|^2},$$
  for all $n\ge 1.$ Corollary \ref{typeIhorosph} then yields that $\{A_n\}$ is non-increasing. \par
  First, we show that $A_n\to 0.$ Indeed, assume instead that $A_n\to\rho>0.$ (\ref{useful100}) implies that there exists $B>0$ such that $B_n<B,$ for all $n\ge 1.$ Also, let $\{M_k\}\subset\mathbb{R}^{+}$ be any increasing sequence tending to $\infty.$ Corollary \ref{typeIhorosph} implies that we can find a decreasing sequence $\{r_k\}$, $r_k\to 1$ such that 
  \begin{equation}\label{useful110}
    A_{n+1}\le\max\bigg\{\frac{A_n}{r_k}, \frac{B_n}{M_k} \bigg\},
  \end{equation}
  for all $n, k\ge 1.$ Let $\epsilon>0$ and choose $k=k_0$ to be such that $B/M_{k_0}<\rho.$ Also, since $r_{k_0}>1$, we can find $N\ge 1$ such that $A_N/r_{k_0}<\rho$. Thus, (\ref{useful110}) yields
  $$A_{N+1}\le \max\bigg\{\frac{A_N}{r_{k_0}}, \frac{B_N}{M_{k_0}} \bigg\}<\rho,$$
  a contradiction. Hence, $A_n\to 0.$ We will show that $B_n\to 0$ as well. Indeed, assume that $B_n\not\to 0$.  (\ref{useful100}) combined with the fact that $A_n\to 0$ implies that $\liminf_nB_n=s>0.$  Also, given $0<K_{-}<K$, Corollary \ref{TypeIIhorosph} yields that for any $t_2>1$ sufficiently close to $1$ one obtains 
  \begin{equation}\label{useful130}
      B_{n+1}\le\max\bigg\{\frac{A_n}{K_{-}}, \frac{B_n}{t_2} \bigg\},
  \end{equation}
  for all $n\ge 1.$ Now, choose $n_0$ such that $A_{n_0}/K_{-}<s/2$ and also $B_{n_0}/t_2<s.$ In view of (\ref{useful130}), we obtain
  $$B_{n_0+1}\le\max\bigg\{\frac{A_{n_0}}{K_{-}}, \frac{B_{n_0}}{t_2} \bigg\}<s,$$
  a contradiction. We conclude that $A_n, B_n\to 0,$ which gives us $F^n=F^n(\lambda_0, \mu_0)\to(\tau^1, \sigma)$. Since $(\lambda_0, \mu_0)$ was arbitrary, we are done.
\end{proof} 
\begin{remark}
We have actually reached the even stronger conclusion that, in the setting of Theorem \ref{TypeITypeIIBBUTNOTC}, the iterates $F^n(\lambda)$ converge to $(\tau^1, \sigma)$ in the horospheric topology, for any $\lambda\in\DD^2.$
\end{remark}
\begin{exmp}\label{exmp1}
Define $\phi, \psi:\DD^2\to\DD$ by 
$$\phi(\lambda)=\frac{1-\lambda^1\lambda^2}{2-\lambda^1-\lambda^2} $$
and 
$$\psi(\lambda)=\begin{cases}
    \frac{(\lambda^2-\lambda^1)-2(1-\lambda^1)(1-\lambda^2)\log\big(\frac{1+\lambda^2}{1-\lambda^2}\frac{1-\lambda^1}{1+\lambda^1}\big)}{(\lambda^2-\lambda^1)+2(1-\lambda^1)(1-\lambda^2)\log\big(\frac{1+\lambda^2}{1-\lambda^2}\frac{1-\lambda^1}{1+\lambda^1}\big)} \hspace{1 cm} \text{ if }\lambda^1\neq\lambda^2, \\
    \frac{-3+5\lambda^1}{5-3\lambda^1} \hspace{5.5 cm} \text{ if }\lambda^1=\lambda^2,
\end{cases}$$
for all $\lambda\in\DD^2$ ($\psi$ has been taken from \cite{McPascoerevisited}). \par Since the slice function $\phi_0$ has $1$ as its Denjoy-Wolff point, Theorem \ref{TypeIcharabridged} implies that the entire closed face $\{1\}\times\text{cl}(\DD)$ consists of B-points for $\phi$ and also $\phi(1, \sigma)=1, $ for all $|\sigma|\le 1.$ Actually, it is easy to see that $\phi$ extends analytically across $(1, \sigma)$ whenever $\sigma\neq 1.$ Now, for $\sigma=1,$ it can be verified that 
$$\frac{D_{-(1, M)}\phi(1, 1)}{-\phi(1, 1)}=-D_{-(1, M)}\phi(1, 1)=\frac{M}{M+1}<1,$$ 
for all $M>0$. Thus, $(1, 1)$ is not a C-point for $\phi$ and also, since \\ $\lim_{M\to\infty}M/(M+1)=1,$ the angular derivative of every slice function $\phi_{\mu}$ at its Denjoy-Wolff point $1$ has to be equal to $1$ (this can be also verified directly, as the slice functions are easy to compute in this case). \par
 Now, we look at $\psi.$ Since $\psi(0, 0)=0,$ $\psi$ is clearly a left (also a right) Type II function. Also, as shown in \cite{McPascoerevisited}, $(1, 1)$ is a B-point for $\psi$ that is not a C-point and $\psi(1, 1)=1.$ We wish to determine whether $(1, 1)$ is also a left Type II DW point for $\psi.$ However, computing the function $\xi:\DD\to\DD$ such that $\psi(\xi(\mu), \mu)=\xi(\mu),$ for all $\mu\in\DD,$ seems impractical here. Instead, we will look at the directional derivatives of $\psi$ at $(1, 1)$ along $\delta_M=(1, M)$ and then use Theorem \ref{ultimateDWchar}. Indeed, in \cite[Section 4]{McPascoerevisited} it was determined that 
 $$K_{(1, 1)}(M)=\frac{D_{-(1, M)}\psi(1, 1)}{-\psi(1, 1)}=-D_{-(1, M)}\psi(1, 1)$$ $$=4M\int_{-1}^1\frac{dt}{(1-t)+(1+t)M}$$
 $$=\begin{cases}
 4\frac{M\ln M}{M-1} \hspace{0.35 cm} \text{ if }M\neq 1,
     \\
   4  \hspace{1.4 cm} \text{ if } M=1.
 \end{cases} $$
Since $K_{(1, 1)}(1)>1$ and $\lim_{M\to 0+}K_{(1, 1)}(M)=0, $ there exists $C>0$ such that $K_{(1, 1)}(C)=1.$ Theorem \ref{ultimateDWchar} then implies that $(1, 1)$ is a left Type II DW point for $\psi$. Also, since $\psi(\lambda^1, \lambda^2)=\psi(\lambda^2, \lambda^1)$, $(1,1)$ must also be a right Type II DW point for $\psi.$
\par 
Now, define $F=(\phi, \psi):\DD^2\to\DD^2.$ In view of our previous observations, we have that $(1, 1)$ is a left Type I DW point for $\phi$ that is not a C-point and it is also a right Type II DW point for $\psi.$ Theorem \ref{TypeITypeIIBBUTNOTC} then allows us to conclude that $F^n\to (1, 1)$ uniformly on compact subsets of $\DD^2.$
\end{exmp}

Before ending this subsection, we remark that the (Type II, Type I) case can be treated in an entirely analogous way.

 \small
\subsection{The (Type I, Type I) case} 
\large

Finally, assume that $\phi$ and $\psi$ are left Type I and right Type I functions, respectively, hence $F=(\phi, \psi)$ does not have any interior fixed points. The following characterization is an easy consequence of Theorem \ref{TypeIcharabridged}, so we omit the proof.
\begin{proposition}\label{(Type1,TypeI)char}
Let $F=(\phi, \psi):\DD^2\to\DD^2$ be holomorphic. Then, $\phi$ and $\psi$ are left Type I and right Type I functions, respectively, if and only if there exists $\tau=(\tau^1, \tau^2)\in\mathbb{T}^2$ that is a left Type I DW point for $\phi$ and a right Type I DW point for $\psi.$
\end{proposition}
Now, let $\tau^1$ and $\tau^2$ be as in Proposition \ref{(Type1,TypeI)char}. In this setting, Herv\'{e} proved that either every cluster point of $\{F^n\}$ will be of the form 
$(\tau^1, h)$, where $h$ is either a holomorphic function $\DD^2\to\DD$ or the constant $\tau^2$, or every cluster point will be of the form 
$(g, \tau^2)$, where $g$ is either a holomorphic function $\DD^2\to\DD$ or the constant $\tau^1$ (Theorem \ref{HERVE's THEOREM}(ii)).  Also, it is not hard to see that in e.g. the former case, there exists a (parabolic) fractional linear transformation $T$ with Denjoy-Wolff point $\tau^2$ such that, whenever both $(\tau^1, h_1)$ and $(\tau^1, h_2)$ appear as \textit{non-constant} cluster points of $\{F^n\}$, it must be true that $h_1=T\circ h_2$ (see the 2nd remark in \cite[Section 14]{HERVE}). Examples showing that these conclusions cannot, in general, be improved are contained in \cite[Section 15]{HERVE}.\par Unfortunately, the proof of Theorem \ref{HERVE's THEOREM}(ii) (to be found in \cite[Sections 12-13]{HERVE}) does not make it clear whether it is possible to determine ``beforehand" which of the two constants ($\tau^1$ or $\tau^2$) will be the one that appears as a coordinate in every cluster point of $\{F^n\}.$ We will show that, under the extra assumption of $(\tau^1, \tau^2)$ not being a C-point for either $\phi$ or $\psi$, one can draw stronger conclusions.  Our proof is independent of Herv\'{e}'s result. 

\begin{proof}[Proof of Theorem \ref{TypeITypeIBBUTNOTC}]
Assume $\tau=(\tau^1, \tau^2)\in\mathbb{T}^2$ satisfies the hypotheses of the theorem. Clearly, $\phi$ and $\psi$ will be left Type I and right Type I functions, respectively. Also, Corollary \ref{TypeIIhorosph} tells us that 
  \begin{equation}\label{useful210}
      F(E(\tau, R_1, R_2))\subset E(\tau, R_1, R_2),
  \end{equation}
  for all $R_1, R_2>0.$ For any fixed $(\lambda_0, \mu_0)\in\DD^2,$ define:
  $$A_n=\frac{|\tau^1-\phi_n(\lambda_0, \mu_0)|^2}{1-|\phi_n(\lambda_0, \mu_0)|^2} \hspace{0.3 cm} \text{ and }\hspace{0.3 cm} B_n=\frac{|\tau^2-\psi_n(\lambda_0, \mu_0)|^2}{1-|\psi_n(\lambda_0, \mu_0)|^2},$$
  for all $n\ge 1.$ (\ref{useful210}) then implies that both $\{A_n\}$ and $\{B_n\}$ are non-increasing. We can then argue as in the proof of Theorem \ref{TypeITypeIIBBUTNOTC} to deduce that $A_n\to 0$ (assuming $\tau$ is not a C-point for $\phi$). Thus, every cluster point of $\{F^n\}$ will be of the form $(\tau^1, h),$ where $h$ is holomorphic on $\DD^2$ and bounded by $1.$ Moreover, since $\{B_n\}$ is bounded, one can deduce that $h$ will have to be either a holomorphic map $\DD^2\to\DD$ or the constant $\tau^2.$
\end{proof}
\begin{remark}
We have actually reached the even stronger conclusion that, in the setting of Theorem \ref{TypeITypeIBBUTNOTC} with e.g. $\tau$ not being a C-point for $\phi$, the points $\phi_n(\lambda)$ converge to $\tau^1$ in the horospheric topology of the unit disk, for any $\lambda\in\DD^2.$
\end{remark}

\begin{exmp} \label{exmp2}
   Define $\phi:\DD^2\to\DD$ by 
   $$\phi(\lambda)=-\frac{3\lambda^1\lambda^2-\lambda^1-\lambda^2-1}{3-\lambda^1-\lambda^2-\lambda^1\lambda^2},$$
   for all $\lambda\in\DD^2$ (this example appears in \cite{SolaTullydynamics}). It can be easily  verified that the Denjoy-Wolff point of the slice function $\phi_0(z)=(z+1)/(3-z)$ is equal to $1.$ Theorem \ref{TypeIcharabridged} then implies that the closed face $\{1\}\times\text{cl}(\DD)$ consists of B-points for $\phi$ and also $\phi(1, \sigma)=1,$ for all $|\sigma|\le 1$. Moreover, we can compute 
   $$\frac{D_{-(1, M)}\phi(1, 1)}{-\phi(1,1)}=-D_{-(1, M)}\phi(1, 1)=\frac{M}{M+1},$$
   for all $M>0.$ Thus, $(1, 1)$ is not a C-point for $\phi$ (and also $\phi'_{\mu}(1)=\lim_{M\to\infty}M/(M+1)=1,$ for all $\mu\in\DD).$ \par
   Now, let $F=(\phi, \psi):\DD^2\to\DD^2,$ where $\psi$ is any (holomorphic) right Type I function such that the Denjoy-Wolff point of all slice functions $\psi(\lambda, \cdot)$ is equal to $1.$ Theorem \ref{TypeITypeIBBUTNOTC} then implies that every cluster point of $\{F^n\}$ will be of the form $(1, h),$ where $h$ is either a holomorphic function $\DD^2\to\DD$ or the constant $1.$ Now, if we take $\psi$ to be e.g. 
   $$\psi(\lambda^1, \lambda^2)=\frac{1-\lambda^1\lambda^2}{2-\lambda^1-\lambda^2},$$
   our observations from Example \ref{exmp1} (and the fact that $\psi(\lambda^1, \lambda^2)=\psi(\lambda^2, \lambda^1)$) show that $(1, 1)$ will be a right Type I DW point for $\psi$ that is not a C-point. Applying Theorem \ref{TypeITypeIBBUTNOTC} again then yields (for this particular choice of $\psi$) that $F^n\to (1, 1)$ uniformly on compact subsets of $\DD^2.$
\end{exmp}

 \small
\section{CONNECTION WITH FROSINI'S WORK} \label{Frosinisection}
\large
Points of Denjoy-Wolff type for holomorphic maps $F:\DD^2\to\DD^2$ have been investigated by Frosini in \cite{FrosiniBusemann}, \cite{FrosiniArxiv}, \cite{FrosiniDynamics}. She defined Denjoy-Wolff points for $F$ as those fixed boundary points where $F$-invariant horospheres are centered, with the exact definition depending on the kind of horospheres in question. In particular, motivated by the definition of ``small" and ``big" horospheres found in \cite{AbateJuliaWolff}, she defined (see \cite[Definitions 3.2-3.3]{FrosiniDynamics}) \textit{quasi-Wolff} and \textit{Wolff} points for $F$ as those fixed boundary points where small horospheres are mapped into big ones and small horospheres are mapped into small ones, respectively. Unfortunately, the existence of quasi-Wolff points is, in general, not very helpful for describing the behavior of $\{F^n\}$,  as big horospheres offer very limited control over the iterates. On the other hand, while Wolff points do offer much more restrictive Julia-type inequalities, they do not always exist (see \cite[Theorem 4.1]{FrosiniDynamics} for a characterization of the set of Wolff points for any self-map $F$ of $\DD^2$). Finally, in \cite[Section 8]{FrosiniBusemann}, Frosini considered \textit{generalized Wolff points}, which motivate our next definition.
\begin{definition}\label{GeneralizedDenjoyWolff}
Let $F=(\phi, \psi):\DD^2\to\DD^2$ be holomorphic with $\tau\in\partial\DD^2$. If there exists $M\in (0, \infty)$ such that 
$$F(E(\tau, R, MR))\subset E(\tau, R, MR),$$
for all $R>0,$ $\tau$ will be called a \textit{generalized Denjoy-Wolff point} for $F.$ 
\end{definition}
As a consequence of Julia's inequality for the bidisk, any generalized Denjoy-Wolff point $\tau\in\partial\DD^2$ of $F$ must be a B-point point for both $\phi$ and $\psi$ such that $F(\tau)=\tau$.  
Notice also that, in contrast to \cite[Definition 33]{FrosiniBusemann}, we do not assume the existence of any complex geodesics, instead relying only on the existence of $F$-invariant ``weighted" horospheres (although the definitions turn out to be equivalent, see Remark \ref{aremark}).  \par 
Let $W(F)$ denote the set of all generalized Denjoy-Wolff points of $F$. Our next result is a slight refinement of \cite[Theorem 39]{FrosiniBusemann}, obtained as a straightforward application of the results developed in this paper. Note that $\tau^1, \tau^2$ will always denote points in $\mathbb{T}$.
\begin{theorem}\label{Frosinibasically}
    Let $F=(\phi, \psi):\DD^2\to\DD^2$ be holomorphic such that $\phi\neq\pi^1$, $\psi\neq\pi^2$ and without any interior fixed points. Then, one and only one of the following three cases is possible:
    \begin{itemize}
        \item[(i)] $W(F)=\{(\tau^1, \tau^2)\}$ if and only if $\phi$ is left Type II and $\psi$ is right Type II;
 \item[(ii)] $\{\tau^1\}\times \DD\subset W(F)\subset\big(\{\tau^1\}\times\DD\big)\cup \{(\tau^1, \tau^2)\}$ (resp., $\DD\times\{\tau^2\}\subset W(F)\subset\big(\DD\times\{\tau^2\}\big)\cup \{(\tau^1, \tau^2)\})$ if and only if $\phi$ is left Type I and $\psi$ is right Type II (resp., $\phi$ is left Type II and $\psi$ is right Type I);
      \item[(iii)]   $W(F)=\big(\{\tau^1\}\times \DD\big)\cup\{(\tau^1, \tau^2)\}\cup \big(\DD\times\{\tau^2\}\big)$ if and only if $\phi$ is left Type I and $\psi$ is right Type I.
    \end{itemize}
\end{theorem}
\begin{proof} Theorem \ref{130} implies that (i)-(iii) contain all possible cases. \par 
    First, assume $\phi$ is left Type II and $\psi$ is right Type II. Theorem \ref{(Type II, Type II)} and Corollary \ref{TypeIIhorosph} imply that $W(F)\supset\{(\tau^1, \tau^2)\}$ for some $\tau^1, \tau^2\in\mathbb{T}$, where $(\tau^1, \tau^2)$ is simultaneously a left
Type II DW point for $\phi$ with constant $M$ and a right Type II DW
point for $\psi$ with constant $1/M.$ Now, assume $(\sigma^1, \sigma^2)\in W(F).$ If either $\sigma^1\in\DD$ or $\sigma^2\in\DD,$ Corollary \ref{typeIhorosph} would imply that either $\psi$ is right Type I or $\phi$ is left Type I, respectively, a contradiction. Thus, $(\sigma^1, \sigma^2)\in\mathbb{T}^2$. But then, Theorems \ref{Juliarevisited} and \ref{TypeIIcharabridged} yield that $(\sigma^1, \sigma^2)$ is simultaneously a left Type II DW point for $\phi$ with constant $M'>0$ and a right Type II DW
point for $\psi$ with constant $1/M'$. In view of Theorem \ref{(Type II, Type II)}, we obtain $(\sigma^1, \sigma^2)=(\tau^1, \tau^2),$ hence $W(F)=\{(\tau^1, \tau^2)\}.$ \par 
Conversely, if $W(F)=\{(\tau^1, \tau^2)\}$, Corollary \ref{typeIhorosph} implies that $\phi$ cannot be a left Type I function and $\psi$ cannot be a right Type I function (else, $W(F)$ would also have to contain facial boundary points). Theorem \ref{130} then yields that $\phi$ is left Type II and $\psi$ is right Type II. \par 
Next, we prove (ii). We will only deal with the (Type I, Type II) version. First, assume that $\phi$ is left Type I and $\psi$ is right Type II, with $\tau^1$ being the common Denjoy-Wolff point of all functions $\phi_{\mu}.$ Corollary \ref{typeIhorosph} implies that $\{\tau^1\}\times \DD\subset W(F)$. If $W(F)=\{\tau^1\}\times \DD$, we are done. Otherwise, assume that we can find a different point $(\sigma^1, \sigma^2)\in W(F).$ We must have $\sigma^1\in\mathbb{T},$ else $\psi$ would be a right Type I function. Also, we may assume $\sigma^2\in\mathbb{T}$ (else we would have $\sigma^1=\tau^1$, in view of Corollary \ref{typeIhorosph}). Now, Theorem \ref{Juliarevisited} (specifically, the fact that (iii) implies (i)) yields that $(\sigma^1, \sigma^2)$ must be either a left Type I or a left Type II DW point for $\phi.$ Proposition \ref{TypeIuniqueness} then tells us that $\sigma^1=\tau^1.$ Note that $(\tau^1, \sigma^2)$ will have to be (in view of Theorem \ref{Juliarevisited}) a right Type II DW point for $\psi.$
Also, if $(t^1, t^2)\in W(F)$ is not contained in $\{\tau^1\}\times\DD$, our previous arguments show that $t^1=\tau^1$ and $(t^1, t^2)$ is, in addition, a right Type II DW point for $\psi$. Proposition \ref{TypeIIuniqueness} then implies $\sigma^2=\tau^2.$ We conclude that $\{\tau^1\}\times \DD\subset W(F)\subset\big(\{\tau^1\}\times\DD\big)\cup \{(\tau^1, \sigma^2)\}$, where $\sigma^2\in\mathbb{T}.$  \par 
Conversely, assume $\{\tau^1\}\times \DD\subset W(F)\subset\big(\{\tau^1\}\times\DD\big)\cup \{(\tau^1, \tau^2)\}$, where $\tau^1, \tau^2\in\mathbb{T}.$ Corollary \ref{typeIhorosph} then implies that $\phi$ is a left Type I and $\psi$ is a right Type II function (else, $W(F)$ would have to contain a face of the form $\DD\times\{\sigma^2\}$), as desired. \par 
Finally, the proof of (iii) rests on Corollary \ref{typeIhorosph}, Proposition \ref{TypeIuniqueness} and Theorem \ref{Juliarevisited}; one can argue in a manner analogous to the proof of (ii). We omit the details. 
\end{proof}
\begin{remark}
As seen in the previous proof, the point $(\tau^1, \tau^2)$ in (ii) will belong to W(F) if and only if it is a right Type II DW point for $\psi.$
\end{remark}

\large

 \par\textit{Acknowledgements}. The second author would like to thank John M\raise.5ex\hbox{c}Carthy and Greg Knese for helpful suggestions.

\printbibliography

\end{document}